\documentclass{article}
\usepackage {amsmath, amsfonts, amssymb, mathrsfs, amsthm,stmaryrd,sectsty,hyphenat,calc,url,subfiles}
\usepackage[shortlabels]{enumitem}
\usepackage{oscar}

\usepackage{amsmath,amssymb,amsxtra,anysize,tikz,combelow,hyperref,graphicx,authblk}

\newcommand{\CA}{\mathcal{A}}
\newcommand{\CCC}{\mathcal{C}}
\newcommand{\CF}{\mathcal{F}}
\newcommand{\CCG}{\mathcal{G}}
\newcommand{\CJ}{\mathcal{J}}
\newcommand{\CL}{\mathcal{L}}
\newcommand{\CK}{\mathcal{K}}
\newcommand{\CO}{\mathcal{O}}
\newcommand{\CP}{\mathcal{P}}
\newcommand{\CT}{\mathcal{T}}
\newcommand{\SBim}{\mathrm{SBim}}
\newcommand{\FHilb}{\mathrm{FHilb}}
\newcommand{\Proj}{\mathrm{Proj}}
\newcommand{\HY}{\mathrm{HY}}
\newcommand{\FT}{\mathrm{FT}}
\newcommand{\Aug}{\mathrm{Aug}}
 
 \begin{title}
{Algebra and geometry of link homology\\
Lecture notes from the IHES 2021 Summer School}
\end{title}
\author[1]{Eugene Gorsky}
\author[2]{Oscar Kivinen}
\author[3]{Jos\'e Simental}

\affil[1]{Department of Mathematics, UC Davis \thanks{egorskiy@math.ucdavis.edu}}
\affil[2]{Department of Mathematics, University of Toronto \thanks{oscar.kivinen@gmail.com}}
\affil[3]{Max Planck Institute for Mathematics, Bonn \thanks{jose@mpim-bonn.mpg.de}}
\begin{document}
\maketitle

\tableofcontents

\section{Introduction}

These notes cover the lectures of the first named author at 2021 IHES Summer School on ``Enumerative Geometry, Physics and Representation Theory'' with additional details and references. They cover the definition of Khovanov-Rozansky triply graded homology, its basic properties and recent advances, as well as three algebro-geometric models for link homology: braid varieties, Hilbert schemes on singular curves and affine Springer fibers, and Hilbert schemes of points on the plane. 

The relations between these models are very subtle and partially conjectural, and yet to be fully understood. Braid varieties can be defined for arbitrary positive braids, and their homology (with the weight filtration) is most clearly related to link homology. By the work of Mellit \cite{Mellit2} they can be used as building blocks for character varieties over punctured surfaces. On the other hand, braid varieties generalize open Richardson and positroid varieties \cite{CGGS2,GL,GL2,GL3} which are important for the study of cluster algebras.

Hilbert schemes on singular curves, compactified Jacobians and affine Springer fibers can be defined for algebraic links, or (more or less equivalently) for matrices with entries in Laurent series. The characteristic polynomial of such a matrix defines a plane curve singularity which intersects a small sphere at a link. The Hilbert schemes of points on a plane curve singularity are closely related to the ``local" version of the Hitchin fibration \cite{MS,MY}, and ``local'' curve-counting invariants \cite{Maulik}. On the other hand, they play an important role in the generalized Springer theory \cite{HKW, GarKiv} for Coulomb branch algebras defined by Braverman, Finkelberg, and Nakajima \cite{BFN,BFN2}.  

A beautiful conjecture of Oblomkov, Rasmussen and Shende \cite{ORS} relates the homology of Hilbert schemes of points  on singular curves to the Khovanov-Rozansky homology of the corresponding links.  By the above, one expects a direct relation between the homology of these Hilbert schemes and of the corresponding braid variety, similar in spirit to the non-abelian Hodge correspondence between the character varieties and the Hitchin system \cite{dCHM}. This also remains a tantalizing conjecture. We review the precise statements and the state of the art for these conjectures below. 

Finally, yet another model for link homology comes from the Hilbert scheme of points on the plane. Roughly speaking, to any braid on $n$ strands one associates a coherent sheaf on $\Hilb^n(\C^2)$ whose cohomology matches the Khovanov-Rozansky homology of the link. Such a sheaf, or rather a complex of sheaves, was constructed by Oblomkov and Rozansky in a series of papers \cite{OR1,OR2,OR3,OR4,OR5,OR6,OR7,OR8}. Another approach \cite{GNR} highlights the importance of the graded algebra generated by the homologies of the powers of the full twist braid. We describe this algebra explicitly following \cite{GH} and explain its relation to the homogeneous coordinate ring of the Hilbert scheme.  This model is also closely related to the ``refined Chern-Simons invariants'' of Aganagic-Shakirov and Cherednik \cite{AS,Cherednik}.

In the case of torus knots, all of the above models yield explicit, yet very different combinatorial descriptions. The Khovanov-Rozansky homology for torus knots was computed in \cite{Mellit}, and the homology of the braid varieties (which in this case coincide with the open positroid strata in the Grassmannians) was computed in \cite{GL}. The homology of the Hilbert schemes on singular curves and affine Springer fibers was computed earlier in numerous papers, starting from \cite{LuSm}. Finally, the sheaves on $\Hilb^n(\C^2)$ for torus knots were constructed in \cite{GN} using the elliptic Hall algebra. The comparison between all these answers is highly nontrivial, and is related to the ``rational Shuffle conjecture'' in combinatorics of Macdonald polynomials, proved in \cite{Mellit3}.

Throughout the notes, we track various structures and homological operations in link homology, and describe their appearance in various models. In particular, we have an action of the polynomial algebra where the number of variables equals the number of link components. From the topological point of view, this corresponds to the action of the homology of the unknot on the homology of an arbitrary link with a chosen marked point. Furthermore, link homology admits a deformation, or ``$y$-ification'' \cite{GH} where a polynomial algebra in an additional set of variables plays an important role. Finally, there is an action of the Lie algebra $\mathfrak{sl}_2$ in the $y$-ified homology \cite{GHM} which implies the symmetry exchanging the two sets of variables. All these structures indicate that the relation between link homology and the geometric models holds on a much deeper categorical level than just isomorphisms of triply graded vector spaces. 

\subsection{Organization of the paper}

The notes are organized as follows. In Section 2 we remind the readers of the basics of knot theory, such as the braid group and theorems of Alexander and Markov on braid closures.

In Section 3 we define Khovanov-Rozansky homology using Soergel bimodules and Rouquier complexes. We describe a method to recursively compute the homology of torus (and other) links, and present many examples.

Section 4 is focused on braid varieties. We define braid varieties, outline their basic properties, and explain their relation to link homology and positroid varieties.

In Section 5 we describe more subtle properties and homological operations in link homology. In particular, we define $y$-ified homology and compute it for all powers of the full twist. We also define ``tautological classes'' in Khovanov-Rozansky homology and use them to outline the proof of the ``$q-t$'' symmetry in this homology. These abstract algebraic constructions are compared to the actual tautological classes in the homology of braid varieties.

In Section 6 we define Hilbert schemes on singular curves, compactified Jacobians, affine Springer fibers, and discuss the relations between them. We state the Oblomkov-Rasmussen-Shende conjecture and discuss the known evidence for it. A connection to geometric representation theory of Coulomb branch algebras is also discussed.

In Section 7 we define the Hilbert scheme of points on the plane and describe its properties, in particular, present it as a symplectic resolution and construct the Procesi bundle. We state an explicit conjecture relating braids to sheaves on the Hilbert scheme, and discuss some examples and approaches to the proof.

\subsection{Further reading}

In these lectures we chose to focus on  topics in link homology most closely related to commutative algebra and algebraic geometry. Unfortunately, this means that we had to skip many other topics of interest, which are also important for understanding the big picture.

In defining link homology, we focus on Soergel bimodules and do not discuss other approaches using webs, foams and categorified quantum groups, referring the reader to \cite{QR,RW} for more details. We do not discuss Khovanov \cite{Khovanov} or $\mathfrak{sl}(N)$ Khovanov-Rozansky homology \cite{KR1} and their relation to HOMFLY homology, and refer to \cite{Rasmussen,RasmussenNotes} instead.

We mostly avoid representation theory and categorification aspects. In particular, we do not define Hecke algebras or work with diagrammatics of Soergel bimodules, and refer the reader to \cite{Soergelbook} for these instead.  We refer to \cite{GORS,GSV} for the connections with the rational Cherednik algebra, and to \cite{GN,Mellit3} for the connections with the DAHA and elliptic Hall algebra.

We also do not discuss very rich combinatorics of $q,t$-Catalan numbers and Macdonald polynomials  \cite{BGSX,BHMPS,GaHa,GHSR,GN,Haglund,HaimanCatalan,Haiman,Haiman2,Mellit3} which deserves to be a subject of a separate course.

Finally, we recommend several other surveys on link homology \cite{Nawata,MBourbaki,Olectures,RasmussenNotes}.

\section*{Acknowledgments}

We would like to thank Andrei Negu\cb{t}, Francesco Sala and Olivier Schiffmann for organizing the IHES Summer School and creating a wonderful working environment despite the COVID-19 pandemic. We thank Mikhail Gorsky, Matt Hogancamp, Anton Mellit, Alexei Oblomkov, Jacob Rasmussen, Minh-Tam Trinh and Paul Wedrich for many discussions over the years and helpful comments about the first draft of these notes. 

We thank all of the participants of the online seminar on homology of torus (and other) knots in Spring 2021, especially Anton Mellit and Pavel Galashin, for very helpful comments and discussions. These lecture notes partly grew from the notes for that seminar \cite{UCDSeminar}.

The work of E. G. was partially supported by the NSF grant DMS-1760329. The work of O.K. was partially supported by the Finnish Academy of Science and Letters. J. S. is grateful for the hospitality and financial support of the Max Planck Institute for Mathematics, where this work was carried out.

\section{Background on link invariants}

In this section we record some basic facts on link invariants.
The {\bf braid group} on $n$ strands has generators $\sigma_1,\ldots,\sigma_{n-1}$ and relations
\begin{equation}
\label{eq: braid}
    \sigma_i\sigma_{i+1}\sigma_i=\sigma_{i+1}\sigma_i\sigma_{i+1},\quad \sigma_i\sigma_j=\sigma_j\sigma_i\ (|i-j|>1).
\end{equation}
We will visualize the generators $\sigma_i$ as positive crossings, and $\sigma_i^{-1}$ as negative crossings:
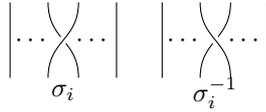
\begin{center}
    \begin{tikzpicture}
    \draw (-0.2,0)--(-0.2,1);
    \draw (0.3,1)..controls (0.3,0.5) and (0.7,0.5)..(0.7,0);
    \draw[white,line width=3] (0.3,0)..controls (0.3,0.5) and (0.7,0.5)..(0.7,1);
    \draw (0.3,0)..controls (0.3,0.5) and (0.7,0.5)..(0.7,1);
    \draw (1.2,0)--(1.2,1);
    \draw (0.1,0.5) node {$\cdots$};
    \draw (0.9,0.5) node {$\cdots$};
    
    \draw (1.8,0)--(1.8,1);
    \draw (2.3,0)..controls (2.3,0.5) and (2.7,0.5)..(2.7,1);
    \draw[white,line width=3] (2.3,1)..controls (2.3,0.5) and (2.7,0.5)..(2.7,0);
    \draw (2.3,1)..controls (2.3,0.5) and (2.7,0.5)..(2.7,0);
    \draw (3.2,0)--(3.2,1);
    \draw (2.1,0.5) node {$\cdots$};
    \draw (2.9,0.5) node {$\cdots$};
    \draw (0.5,-0.2) node {$\sigma_i$};
    \draw (2.5,-0.2) node {$\sigma_i^{-1}$};
    \end{tikzpicture}
\end{center}
The strands in a braid are labeled from $1$ to $n$, and the composition is given by vertical stacking. The following theorems \cite{Alex,Birman,Markov} relate links and braids.

\begin{theorem}[Alexander]
Any link can be obtained as a closure of some braid.
\end{theorem}

\begin{theorem}[Markov]
Two braid closures represent the same link if and only if the braids are related by a sequence of the following moves:
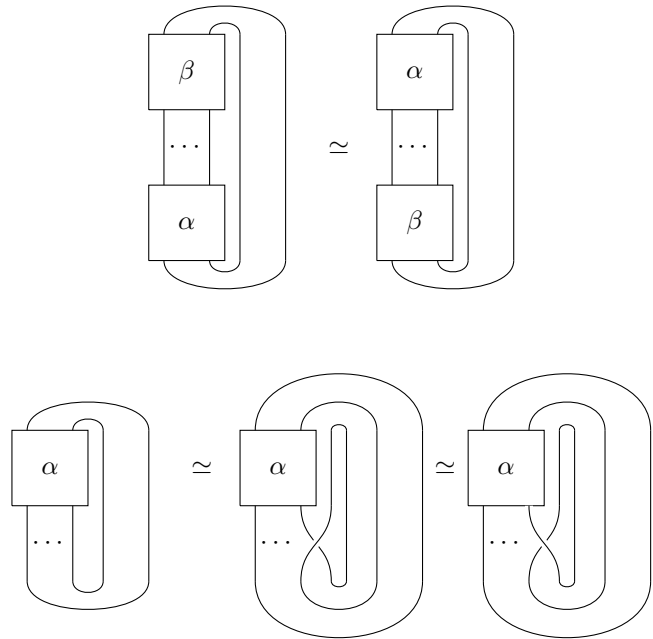
\begin{center}
    \begin{tikzpicture}
      \draw (0,0)--(0,1)--(1,1)--(1,0)--(0,0);
      \draw (0,2)--(0,3)--(1,3)--(1,2)--(0,2);
      \draw (0.2,1)--(0.2,2);
      \draw (0.8,1)--(0.8,2);
      \draw (0.5,1.5) node {$\cdots$};
      \draw (0.5,0.5) node {$\alpha$};
      \draw (0.5,2.5) node {$\beta$};
      \draw (0.2,0)..controls (0.2,-0.5) and (1.8,-0.5)..(1.8,0);
      \draw (0.8,0)..controls (0.8,-0.2) and (1.2,-0.2)..(1.2,0);
      \draw (0.2,3)..controls (0.2,3.5) and (1.8,3.5)..(1.8,3);
      \draw (0.8,3)..controls (0.8,3.2) and (1.2,3.2)..(1.2,3);
      \draw (1.2,0)--(1.2,3);
      \draw (1.8,0)--(1.8,3);
      \draw (2.5,1.5) node {$\simeq$};
      \draw (3,0)--(3,1)--(4,1)--(4,0)--(3,0);
      \draw (3,2)--(3,3)--(4,3)--(4,2)--(3,2);
      \draw (3.2,1)--(3.2,2);
      \draw (3.8,1)--(3.8,2);
      \draw (3.5,1.5) node {$\cdots$};
      \draw (3.5,0.5) node {$\beta$};
      \draw (3.5,2.5) node {$\alpha$};
      \draw (3.2,0)..controls (3.2,-0.5) and (4.8,-0.5)..(4.8,0);
      \draw (3.8,0)..controls (3.8,-0.2) and (4.2,-0.2)..(4.2,0);
      \draw (3.2,3)..controls (3.2,3.5) and (4.8,3.5)..(4.8,3);
      \draw (3.8,3)..controls (3.8,3.2) and (4.2,3.2)..(4.2,3);
      \draw (4.2,0)--(4.2,3);
      \draw (4.8,0)--(4.8,3);
    \end{tikzpicture}
\end{center}
\begin{center}
    \begin{tikzpicture}
    \draw (0,0)--(0,1)--(1,1)--(1,0)--(0,0);
    \draw (0.5,0.5) node {$\alpha$};
    \draw (0.2,-1)--(0.2,0);
    \draw (0.8,-1)--(0.8,0);
    \draw (0.5,-0.5) node {$\cdots$};
    \draw (0.2,-1)..controls (0.2,-1.5) and (1.8,-1.5)..(1.8,-1);
    \draw (0.8,-1)..controls (0.8,-1.2) and (1.2,-1.2)..(1.2,-1);
    \draw (0.2,1)..controls (0.2,1.5) and (1.8,1.5)..(1.8,1);
    \draw (0.8,1)..controls (0.8,1.2) and (1.2,1.2)..(1.2,1);
    \draw (1.2,-1)--(1.2,1);
    \draw (1.8,-1)--(1.8,1);
    
    \draw (2.5,0.5) node {$\simeq$};
    
    \draw (3,0)--(3,1)--(4,1)--(4,0)--(3,0);
    \draw (3.5,0.5) node {$\alpha$};
    \draw (3.2,-1)--(3.2,0);
    \draw (3.8,0)..controls (3.8,-0.5) and (4.2,-0.5)..(4.2,-1);
    \draw[white,line width=3] (3.8,-1)..controls (3.8,-0.5) and (4.2,-0.5)..(4.2,0);
    \draw (3.8,-1)..controls (3.8,-0.5) and (4.2,-0.5)..(4.2,0);
    \draw (4.2,0)--(4.2,1);
    \draw (3.5,-0.5) node {$\cdots$};
    \draw (3.2,-1)..controls (3.2,-2) and (5.4,-2)..(5.4,-1);
    \draw (3.8,-1)..controls (3.8,-1.5) and (4.8,-1.5)..(4.8,-1);
    \draw (4.2,-1)..controls (4.2,-1.1) and (4.4,-1.1)..(4.4,-1);
    \draw (3.2,1)..controls (3.2,2) and (5.4,2)..(5.4,1);
    \draw (3.8,1)..controls (3.8,1.5) and (4.8,1.5)..(4.8,1);
    \draw (4.2,1)..controls (4.2,1.1) and (4.4,1.1)..(4.4,1);
    \draw (4.4,-1)--(4.4,1);
    \draw (4.8,-1)--(4.8,1);
     \draw (5.4,-1)--(5.4,1);
     
     \draw (5.7,0.5) node {$\simeq$};
     
     \draw (6,0)--(6,1)--(7,1)--(7,0)--(6,0);
    \draw (6.5,0.5) node {$\alpha$};
    \draw (6.2,-1)--(6.2,0);
    \draw (6.8,-1)..controls (6.8,-0.5) and (7.2,-0.5)..(7.2,0);
    \draw[white,line width=3] (6.8,0)..controls (6.8,-0.5) and (7.2,-0.5)..(7.2,-1);
    \draw (6.8,0)..controls (6.8,-0.5) and (7.2,-0.5)..(7.2,-1);
    \draw (7.2,0)--(7.2,1);
    \draw (6.5,-0.5) node {$\cdots$};
    \draw (6.2,-1)..controls (6.2,-2) and (8.4,-2)..(8.4,-1);
    \draw (6.8,-1)..controls (6.8,-1.5) and (7.8,-1.5)..(7.8,-1);
    \draw (7.2,-1)..controls (7.2,-1.1) and (7.4,-1.1)..(7.4,-1);
    \draw (6.2,1)..controls (6.2,2) and (8.4,2)..(8.4,1);
    \draw (6.8,1)..controls (6.8,1.5) and (7.8,1.5)..(7.8,1);
    \draw (7.2,1)..controls (7.2,1.1) and (7.4,1.1)..(7.4,1);
    \draw (7.4,-1)--(7.4,1);
    \draw (7.8,-1)--(7.8,1);
     \draw (8.4,-1)--(8.4,1);
    \end{tikzpicture}
\end{center}
\end{theorem}

We will refer to these moves as to conjugation and (respectively positive and negative) stabilization, or collectively as Markov moves. Using these theorems, we can sketch a possible strategy of constructing link invariants as follows:
\begin{itemize}
    \item Assign some objects to crossings $\sigma_i^{\pm}$.
    \item Verify braid relations \eqref{eq: braid}. This would yield a braid invariant.
    \item Describe an operation for closing a braid.
    \item Verify that the result is invariant under Markov moves (conjugation and stabilization).
\end{itemize}

Most of link invariants described in these lectures will follow this strategy. Sometimes it will be helpful to consider slightly weaker invariants for braid closures, which do not change only under conjugation, or under conjugation and positive stabilization.

\section{Khovanov-Rozansky homology: definitions and computations}

All link homologies and most of the computations in these notes can be defined with integer coefficients. For the interest of clarity and various technical simplifications, we work over $\C$ instead.
% We will comment on some known results over $\Z$.

\subsection{Soergel bimodules and Rouquier complexes}

Consider the ring $R=\C[x_1,\ldots,x_n]$ with the action of the symmetric group $S_n$ which permutes the variables. We will consider various rings of invariants, most importantly the rings $R^{s_{i}}$ of polynomials that are invariant under the transposition $x_{i} \leftrightarrow x_{i+1}$, $i = 1, \dots, n-1$. We will work with $R-R$ bimodules which we alternatively interpret as modules over $\C[x_1,\ldots,x_n,x'_1,\ldots,x'_n]$ where the left action of $R$ corresponds to the action of $x_i$, and the right action corresponds to the action of $x'_i$. Given two $R-R$ bimodules $M$ and $N$, we can consider their tensor product $M\otimes_R N$. The left action of $R$ on $M, N$ will be denoted by $x_i$, $x'_i$, respectively, and the right action will be denoted by $x'_i, x''_i$, respectively. Note that this is consistent with the relations on the tensor product $M \otimes_R N$.

\begin{remark}
To ease the notation, we will sometimes write $MN$ for $M \otimes_{R} N$. If, on the other hand, we are taking the tensor product over a ring other than $R$, we will always indicate this in the notation.
\end{remark}%By convention, we will assume that $x_i$ and $x'_i$ act on $M$, while $x'_i$ and $x''_i$ act on $N$.

Note that the ring $R$ is graded, with $\deg(x_i) = 2$ for every $i = 1, \dots, n$. We will work with graded $R$-bimodules, so we have a decomposition $M = \bigoplus_{j \in \Z} M_{j}$ with $x_{i}M_{j}, M_{j}x_{i} \subseteq M_{j+2}$. The tensor product of graded bimodules is naturally graded, and we will denote the grading shift by $(1)$, so that $M(1)_{i} = M_{i+1}$. 

\begin{remark}
Note that, under this convention for the grading shift, if the graded dimension of $M$ is $\operatorname{gdim}(M) = \sum_{i} Q^{i}\dim M_{i}$, then $\operatorname{gdim}(M(1)) = Q^{-1}\operatorname{gdim}(M)$. 
\end{remark}

For us, the most important bimodules are
$$
B_i:=R\otimes_{R^{s_i}} R(1)=\frac{\C[x_1,\ldots,x_n,x'_1,\ldots,x'_n]}{\left(x_i+x_{i+1}=x'_i+x'_{i+1},\ x_ix_{i+1}=x'_ix'_{i+1},\ x_j=x'_j\ (j\neq i,i+1)\right)}
$$
where $s_i=(i\ i+1)$ and $i$ runs from $1$ to $(n-1)$. Note that, due to the grading shift, the degree of $x_{j} \in B_{i}$ is $1$ for every $j$. Likewise, the degree of $1 \in B_{i}$ is $-1$. 

\begin{definition}
The category $\SBim_n$ of Soergel bimodules is the smallest full subcategory of the category of graded $R-R$-bimodules containing $R$ and $B_i$ and closed under direct sums, grading shifts, tensor products and direct summands.
\end{definition}

\begin{lemma}
\label{lem: Soergel quadratic}
We have
\begin{equation}
\label{eq: Soergel quadratic}
    B_i\otimes_{R} B_i\simeq B_i(1)\oplus B_i(-1).
\end{equation}
\end{lemma}
\begin{proof}
Let $s=(i\; i+1)$. We have 
$$B_i\otimes_R B_i= (R\otimes_{R^s} R(1)) \otimes_R (R\otimes_{R^s} R(1))=R\otimes_{R^s} R\otimes_{R^s} R(2)$$
Decompose $R=R^s\oplus R^{\epsilon}$ where $\epsilon$ denotes the $s$-alternating part. As graded $R_s$-bimodules, $R^\epsilon\cong R^s(-2)$ (the isomorphism divides $p \in R^\epsilon$ by $x_i-x_{i+1}$ or more generally by $\alpha_s$).
Therefore, we obtain
\begin{align*}
R\otimes_{R^s} (R^s\oplus R^\epsilon)\otimes_{R^s} R(2) &= \left[R\otimes_{R^s} R(2)\right] \oplus \left[R\otimes_{R^s} R^{s}(-2) \otimes_{R^s} R(2)\right]\\
 &= B_i(1)\oplus B_i(-1)
\end{align*}
%Note that this computation works for any Coxeter group. How does the computation change for {\em graded} $R - R$-bimodules?
\end{proof}

\begin{example}
One can also check the equation
\begin{equation}
\label{eq: Soergel braid}
    B_i\otimes_{R}B_{i+1}\otimes_{R} B_i\simeq B_i\oplus B_{i,i+1},\ \text{where}\ B_{i,i+1}=R\otimes_{R^{s_i,s_{i+1}}} R(3).
\end{equation}
Note that $s_i$ and $s_{i+1}$ generate a subgroup in $S_n$ isomorphic to $S_3$.
\end{example}

\begin{remark}
The category of Soergel bimodules can be defined for any Coxeter group. The analogues of $B_i$ correspond to simple reflections, the analogue of \eqref{eq: Soergel quadratic} holds on the nose  and the analogue of \eqref{eq: Soergel braid} holds with certain modifications. We refer to \cite{Soergelbook} for more details and references on Soergel bimodules.
\end{remark}

We will be working with the homotopy category of (bounded) complexes of Soergel bimodules which we denote by $\CK_n:=\CK^b(\SBim_n)$. 

\begin{remark}
 The category of Soergel bimodules is additive but not abelian. This means that the homotopy category $\CK^b(\SBim_n)$ still makes sense, but one cannot a priori define a derived category of Soergel bimodules. In principle, one can consider a subcategory of the derived category of all $R-R$ bimodules generated by Soergel bimodules, but this would lead to lots of confusion and incorrect answers. For example, in the derived category the complexes of \eqref{eq: rouquier} are isomorphic up to a grading shift.
 
 This situation might be compared to the construction of the derived category as the homotopy category of projective modules. Indeed, by the work of Soergel \cite{Soe} the category of Soergel bimodules is closely related to the Bernstein-Gelfand-Gelfand category $\mathcal{O}$, and Soergel modules correspond to projective objects in that category.
\end{remark}

Next, we describe some morphisms between Soergel bimodules.

\begin{lemma}
\label{lem: problem 2}
(a) There is a natural projection from $B_i(-1)$ to $R$ which sends $1$ to $1$. 

(b) There is a well defined morphism of bimodules $R\to B_i(1)$ which sends $1$ to $x_i-x'_{i+1}$.
\end{lemma}

\begin{proof}
Any $R-R$-bimodule homomorphism $B_i=R\otimes_{R^s} R\to R$ needs to send the bimodule generator $1\otimes 1$ somewhere in $R$. After fixing this, bilinearity forces the map to be an $R$-multiple of $f\otimes g \mapsto fg$. Finally, since $1 \in B_{i}$ has degree $-1$ while $1 \in R$ has degree $0$, we need the shift $B_i(-1)$ to have a map of graded bimodules.

Similarly, ignoring gradings for the time being, a map $\mu: R\to B_i$ needs to send $\mu:  1\mapsto \sum_i a_i\otimes_{R^s} b_i$ for some $a_i, b_i$. For this to be a bimodule homomorphism, 
we need $p\mu(1)=\mu(1)p$ since $R$ is commutative. Decomposing $p=p^s+\alpha_s p^\epsilon$ we get
$$\sum_i a_i \otimes b_i p^s+
\alpha_s a_i\otimes b_i p^\epsilon=
\sum_i a_i \otimes b_i p^s+a_i\otimes b_i\alpha_s p^\epsilon
$$ meaning that $1\mapsto 1\otimes \alpha_s + \alpha_s \otimes 1$ works. For the $W=S_n$ case, we have $\alpha_s=2(x_i-x_{i+1}')$. That multiples of $\mu$ are all the homomorphisms can be seen e.g. from the tensor-hom adjunction for bimodules and the first part.

Concretely, let us prove that there is a map of bimodules $R\to B_i$ which sends $1$ to $x_i-x'_{i+1}$. We need to check that it sends the defining ideal for $R$ inside the defining ideal for $B_i$. Indeed,
$$
(x_i-x'_{i+1})(x_i-x'_i)=\textrm{(symmetric in\ } x'_i,\ x'_{i+1})=
(x_i-x_{i+1})(x_i-x_i)=0,
$$
similarly
$$
(x_i-x'_{i+1})(x_{i+1}-x'_{i+1})=(x'_{i}-x'_{i+1})(x'_{i+1}-x'_{i+1})=0.
$$
Note that 
$$
1\otimes \alpha_s + \alpha_s \otimes 1=(x_i-x_{i+1})+(x'_{i}-x'_{i+1})=
$$
$$
2x_i-(x_i+x_{i+1})+(x'_{i}+x'_{i+1})-2x'_{i+1}=2(x_i-x'_{i+1}),
$$
so the two solutions agree up to a scalar. Finally, note that  since $x_i - x_{i+1}'$ has degree $1$ in $B_i$, we need to shift the degree on $B_i$ in order to make this a map of graded $R$-bimodules.
\end{proof}

\begin{remark}\label{rmk:hom grading}
Note that we have $\Hom_{R\text{-bimod}}(R, B_i) = R(-1)$, since the map $1 \to x_i - x'_{i+1}$ has degree $1$. 
\end{remark}

Using the above maps between $B_i$ and $R$, we can define {\bf Rouquier complexes} as their cones:
\begin{equation}
\label{eq: rouquier}
    T_i:=[B_i(-1)\to \underline{R}],\quad T_i^{-1}=[\underline{R}\to B_i(1)].
\end{equation}
Here the underlined terms are in the homological degree zero and the maps are defined in Lemma \ref{lem: problem 2} above.
The following is the fundamental result of Rouquier \cite{Rouquier}:
\begin{theorem}[\cite{Rouquier}]
\label{th: rouquier}
The complexes $T_i$ and $T_i^{-1}$ satisfy braid relations up to homotopy:
$$
T_i\otimes_R T_i^{-1}\simeq R,\quad T_i\otimes_R T_{i+1}\otimes_R T_i\simeq T_{i+1}\otimes_R T_i\otimes_R T_{i+1},\quad T_i\otimes_R T_j=T_j\otimes_R T_i,\ (|i-j|>1).
$$
\end{theorem}

The proof of the first relation is sketched as a lemma below. The second equation can be proved similarly using \eqref{eq: Soergel braid}, and the last equation is obvious.

\begin{lemma}\label{lem: tt-1}
The complex $T_i\otimes_{R} T_i^{-1}$ is homotopy equivalent to $R$.
\end{lemma}

\begin{proof}
The tensor product of complexes is 
\begin{align*}
    [B_i(-1) \xrightarrow{m} & \underline{R} ] \otimes_R \\ 
    [& \underline{R} \xrightarrow{\Delta} B_i(1)] \\
    = [B_i(-1) \xrightarrow{m\oplus \delta } &\underline{R \oplus B_i\otimes B_i} \xrightarrow{\Delta\oplus \mu}  B_i(1)]
\end{align*}
for $\delta=\id \otimes \Delta\circ m$ and $\mu=m\otimes \id $. Recall from Lemma \ref{lem: Soergel quadratic} that $B_i \otimes_{R} B_i\cong B_i(1)\oplus B_i(-1)$. This gives a subcomplex $[B_i(-1) \to B_i(-1)\oplus B_i(1) \to B_i(1)]\cong 0$ (with differentials as above) leaving us with $$T_i\otimes_R T^{-1}_i \cong [0\to \underline{R}\to 0]=R.$$
\end{proof}

Given a braid $\beta=\sigma_{i_1}^{\epsilon_1}\cdots \sigma_{i_r}^{\epsilon_r}$ where $\sigma_i$ are the braid group generators and $\epsilon_i=\pm 1$, we can define the corresponding {\bf Rouquier complex}
$$
T_{\beta}:=T_{i_1}^{\epsilon_1}\otimes_R\cdots \otimes_R T_{i_r}^{\epsilon_r}.
$$
By Theorem \ref{th: rouquier} this complex is a well defined object in the homotopy category $\CK_n$.

\begin{example} 
\label{ex: two strand}
We have
$$T_i^2=[B_i(-1)\oplus B_i(-3) \to B_i(-1)\oplus B_i(-1) \to \underline{R}]\cong [B_i(-3)\to B_i(-1)\to R]$$
and claim $$T_i^k\cong [\underbrace{B_i(-2k+1)\to B_i(-2k+3)\to \cdots \to B_i(-1)}_k \to \underline{R}]$$
Indeed, if $T_i^{k-1}$ is such then, decomposing $T_i^{k} = T_i^{k-1}T_i$ we get
$$\left[B_iB_i(-2k+2)\to \begin{matrix} B_i(-2k+3) \\ \oplus \\ B_iB_i(-2k+4) \end{matrix}\to \begin{matrix} B_i(-2k+5) \\ \oplus \\ B_iB_i(-2k+6)  \end{matrix} \to \cdots  \begin{matrix} B_i(-3) \\ \oplus \\ B_iB_i(-2) \end{matrix} \to \begin{matrix}B_i(-1) \\ \oplus \\ B_i(-1) \end{matrix} \to \underline{R}\right].$$
Using $B_iB_i\cong B_i(1)\oplus B_i(-1)$ and the form of the differentials above, this simplifies to what we want.
\end{example}

\subsection{Khovanov-Rozansky homology}

Next, we define the operation corresponding to the braid closure. If $M$ is an $R-R$ bimodule, we define its Hochschild cohomology as 
$$
\HH^i(M):=\Ext^i_{R\operatorname{-bimod}}(R,M).
$$
Given a complex $M_{\bullet}=(M_k,d)$ of $R-R$ bimodules (in particular, of Soergel bimodules), we define complexes
$$
\HH^i(M_{\bullet})=(\HH^i(M_k),d_i)
$$
where $d_i$ is the differential induced by $d$. In other words, we apply the functor $\HH^i(-)$ separately for each $i$, and term-wise in $M_{\bullet}$. The output is a collection of complexes of $R$-modules, one for each Hochschild degree.

\begin{remark}
More abstractly, for each $i$ $\HH^i(-)$ defines an additive functor on the category of $R-R$ bimodules, and hence an additive functor on the category of Soergel bimodules.  We extend this functor to the homotopy category $\CK_n$.
\end{remark}

\begin{remark}
The definition of $\HH^i(-)$ might appear a bit unnatural from the viewpoint of Soergel category. This issue is resolved in \cite{BT} where it is proved that the functors $\HH^i(-)$ are representable, that is, there are certain explicit complexes of Soergel bimodules $W_i$ such that $\Hom(W_i,-)\simeq \HH^i(-)$. 
\end{remark}

In particular, for $i=0$ we get $\HH^0(M)=\Hom_{R\operatorname{-bimod}}(R,M)$ for a bimodule $M$ and $$\HH^0(M_{\bullet})=\Hom_{R\operatorname{-bimod}}(R,M_{\bullet})$$ for a complex $M_{\bullet}$. Here we regard $\Hom$ between two complexes as a complex in a standard way. 

\begin{definition}
The Khovanov-Rozansky homology of the braid $\beta$ is defined as the homology of the Hochschild homology of the Rouquier complex $T_{\beta}$:
$$
\HHH(\beta)=H(\HH(T_{\beta})).
$$
\end{definition}

The Khovanov-Rozansky homology is {\bf triply graded}:
\begin{itemize}
    \item The  $Q$-grading corresponds to the internal grading on Soergel bimodules where all $x_i$ have degree 2. Note that all morphisms in definitions of $T_i$ and $T_i^{-1}$ are homogeneous, and the equations \eqref{eq: Soergel quadratic} and \eqref{eq: Soergel braid} hold with appropriate grading shifts. %\cite{KhSoergel,Rouquier,Soergelbook}
    \item The $T$-grading is the homological grading in the Rouquier complex $T_{\beta}$.
    \item The $A$-grading is the Hochschild degree which equals $i$ for $\HH^i$.
\end{itemize}

\begin{theorem}[\cite{KR2,KhSoergel}]
The Khovanov-Rozansky homology $\HHH(\beta)$ is a link invariant, up to an overall grading shift. More precisely, $\HHH(\beta)$ is invariant under conjugation and positive stabilization, while negative stabilization shifts it up by one $A$--degree.
\end{theorem}

\begin{remark}
It is possible to fix the ambiguity of grading shift and get an honest link invariant, see e.g. \cite{Wu}.
\end{remark}

For most of these notes, we will focus on the $A=0$ part of Khovanov-Rozansky homology, which corresponds to $\HH^0$. It is invariant under conjugation and positive stabilization, but vanishes after a single negative stabilization.

\begin{example}
\label{ex: two strand homology}
Continuing Example \ref{ex: two strand}, we can apply $\HH^0=\Hom(R,-)$ term-wise and obtain: 
$$
\Hom(R, T_i^k)\cong [\underbrace{R(-2k) \to R(-2k+2)\to \cdots \to R(-2) \to \underline{R}}_{k+1}],
$$ 
\noindent cf. Remark \ref{rmk:hom grading}. The differentials in $T_i^{k}$ alternate between $x_i-x'_i$ and $x_i-x'_{i+1}$ (so that $(x_i-x'_i)(x_i-x'_{i+1})=0$ as in Lemma \ref{lem: problem 2}), hence the differentials in $\Hom(R, T_i^k)$ alternate between $x_i-x_i=0$ and $x_i-x_{i+1}$. For example,
$$
\Hom(R,T_i^2)=[R(-4)\xrightarrow{0}R(-2)\xrightarrow{x_i-x_{i+1}} \underline{R}]
$$
and
$$
\Hom(R,T_i^3)=[R(-6)\xrightarrow{x_i-x_{i+1}}R(-4)\xrightarrow{0} R(-2)\xrightarrow{x_i-x_{i+1}}\underline{R}]
$$
One can easily compute the homology of the resulting complex and obtain that the Poincar\'e polynomials of the $A = 0$ part of $\HHH(T(2,2))$  and $\HHH(T(2,3))$ are
$$
\HHH^{A=0}(T(2,2)) = \frac{Q^{4}T^{-2}}{(1 - Q^2)^2} + \frac{1}{1 - Q^2}, \quad \HHH^{A=0}(T(2,3)) = \frac{1 + Q^4T^{-2}}{1 - Q^2}
$$
\end{example}

%\subsection{Basic properties}

\subsection{Recursions and parity}

As one can see from the definition, the complex $T_{\beta}$ grows exponentially in the number of crossings in a braid, which quickly makes the direct computation of Khovanov-Rozansky homology, even with a help of a computer, unfeasible. This was a major stumbling block in link homology for over a decade until significant progress was obtained in a series of papers of Elias, Hogancamp and Mellit \cite{EH,Hog2,Mellit,HM}. This culminates in the following:

\begin{theorem}[\cite{HM}]
\label{thm: parity}
The Khovanov-Rozansky homology of all {\bf positive} torus links $T(m,n)$ is supported in {\bf even} homological degrees and the corresponding Poincar\'e polynomial can be computed using an explicit recursion.
\end{theorem}

\begin{example}
The Poincar\'e polynomial for the $A=0$ part of $\HHH(T(n,n+1))$ is given by the $q,t$--Catalan polynomial defined by Garsia and Haiman in \cite{GaHa}.
\end{example}

\begin{remark}
Theorem \ref{thm: parity} confirms a series of conjectures about the combinatorics of $\HHH(T(m,n))$ proposed in \cite{G,GN,GORS,ORS}.
\end{remark}

Let us describe the idea behind the proof of this theorem. Hogancamp in \cite{Hog} observed that for each $n$ there exists a complex of Soergel bimodules $K_n$ satisfying the following relations:

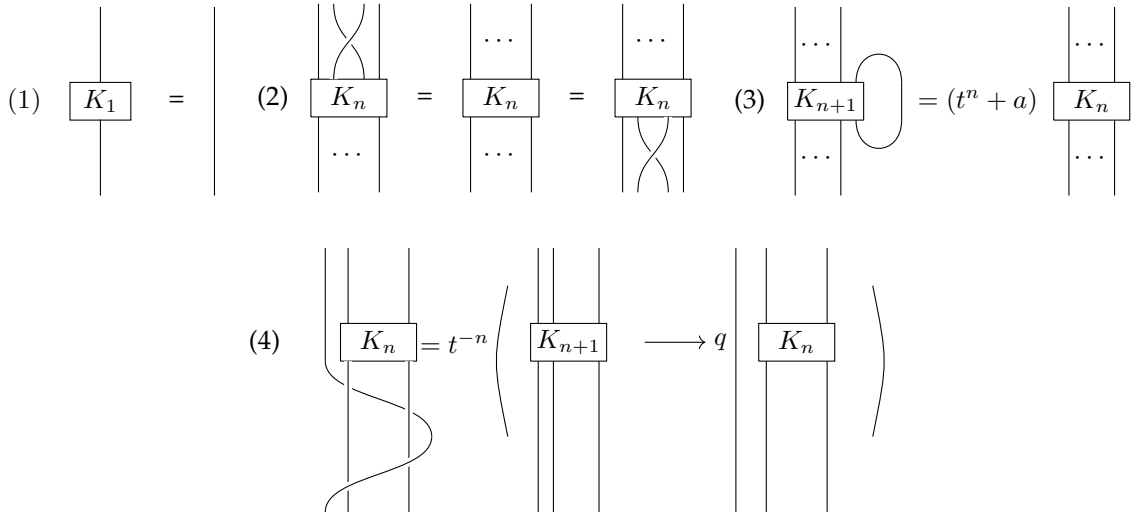
\begin{center}
\begin{tikzpicture}
\draw (-1,1.25) node {$(1)$};
\draw (0,0)--(0,1)--(-0.4,1)--(-0.4,1.5)--(0.4,1.5)--(0.4,1)--(0,1);
\draw (0,1.5)--(0,2.5);
\draw (0,1.25) node {$K_1$};
\draw (1,1.25) node {=};
\draw (1.5,0)--(1.5,2.5);
\end{tikzpicture}
\quad
\begin{tikzpicture}
\draw (-1,1.25) node {(2)};
\draw  (-0.5,1)--(-0.5,1.5)--(0.5,1.5)--(0.5,1)--(-0.5,1);
\draw (-0.4,0)--(-0.4,1);
\draw (0,0.5) node {$\cdots$};
\draw (0.4,0)--(0.4,1);
\draw (0,1.25) node {$K_n$};
\draw (-0.4,1.5)--(-0.4,2.5);
\draw (0.4,1.5)--(0.4,2.5);
\draw (-0.2,2.5)..controls (-0.2,2) and (0.2,2)..(0.2,1.5);
\draw[white,line width=3] (-0.2,1.5)..controls (-0.2,2) and (0.2,2)..(0.2,2.5);
\draw (-0.2,1.5)..controls (-0.2,2) and (0.2,2)..(0.2,2.5);
\draw (1,1.25) node {=};
\draw  (1.5,1)--(1.5,1.5)--(2.5,1.5)--(2.5,1)--(1.5,1);
\draw (1.6,0)--(1.6,1);
\draw (2,0.5) node {$\cdots$};
\draw (2.4,0)--(2.4,1);
\draw (2,1.25) node {$K_n$};
\draw (1.6,1.5)--(1.6,2.5);
\draw (2.4,1.5)--(2.4,2.5);
\draw (2,2) node {$\cdots$};
\draw (3,1.25) node {=};
\draw  (3.5,1)--(3.5,1.5)--(4.5,1.5)--(4.5,1)--(3.5,1);
\draw (3.6,0)--(3.6,1);
\draw (4,2) node {$\cdots$};
\draw (4.4,0)--(4.4,1);
\draw (4,1.25) node {$K_n$};
\draw (3.6,1.5)--(3.6,2.5);
\draw (4.4,1.5)--(4.4,2.5);
\draw (3.8,1)..controls (3.8,0.5) and (4.2,0.5)..(4.2,0);
\draw[white,line width=3] (3.8,0)..controls (3.8,0.5) and (4.2,0.5)..(4.2,1);
\draw (3.8,0)..controls (3.8,0.5) and (4.2,0.5)..(4.2,1);
\end{tikzpicture}
\quad
\begin{tikzpicture}
\draw (-1,1.25) node {(3)};
\draw  (-0.5,1)--(-0.5,1.5)--(0.5,1.5)--(0.5,1)--(-0.5,1);
\draw (-0.4,0)--(-0.4,1);
\draw (-0.1,0.5) node {$\cdots$};
\draw (-0.1,2) node {$\cdots$};
\draw (0,1.25) node {$K_{n+1}$};
\draw (-0.4,1.5)--(-0.4,2.5);
\draw (0.4,1.5)..controls (0.4,2) and (1,2)..(1,1.5);
\draw (0.4,1)..controls (0.4,0.5) and (1,0.5)..(1,1);
\draw (0.2,0)--(0.2,1);
\draw (0.2,1.5)--(0.2,2.5);
\draw (1,1)--(1,1.5);
\draw (2,1.25) node {$=(t^n+a)$};
\draw  (3,1)--(3,1.5)--(4,1.5)--(4,1)--(3,1);
\draw (3.2,0)--(3.2,1);
\draw (3.8,0)--(3.8,1);
\draw (3.5,0.5) node {$\cdots$};
\draw (3.5,2) node {$\cdots$};
\draw (3.5,1.25) node {$K_{n}$};
\draw (3.2,1.5)--(3.2,2.5);
\draw (3.8,1.5)--(3.8,2.5);
\end{tikzpicture}
\end{center}

\begin{center}
\begin{tikzpicture}
\draw (-1.5,1.25) node {(4)};
\draw  (-0.5,1)--(-0.5,1.5)--(0.5,1.5)--(0.5,1)--(-0.5,1);
\draw (0,1.25) node {$K_{n}$};
\draw  (-0.7,2.5)--(-0.7,1);
\draw (-0.4,1.5)--(-0.4,2.5);
\draw (0.4,1.5)--(0.4,2.5);
\draw  (-0.7,1)..controls (-0.7,0.5) and (0.7,0.5)..(0.7,0);
\draw[white,line width=3] (-0.4,1)--(-0.4,0);
\draw [white,line width=3]  (0.4,1)--(0.4,0);
\draw  (-0.4,1)--(-0.4,-1);
\draw    (0.4,1)--(0.4,-1);
\draw [white,line width=3]    (0.7,0)..controls (0.7,-0.5) and (-0.7,-0.5)..(-0.7,-1);
\draw  (0.7,0)..controls (0.7,-0.5) and (-0.7,-0.5)..(-0.7,-1);

\draw (1,1.25) node {$=t^{-n}$};
\draw (1.7,2)..controls (1.5,1) and (1.5,1)..(1.7,0);
\draw  (2,1)--(2,1.5)--(3,1.5)--(3,1)--(2,1);
\draw (2.5,1.25) node {$K_{n+1}$};
\draw (2.1,1.5)--(2.1,2.5);
\draw (2.3,1.5)--(2.3,2.5);
\draw (2.9,1.5)--(2.9,2.5);
\draw (2.1,1)--(2.1,-1);
\draw (2.3,1)--(2.3,-1);
\draw (2.9,1)--(2.9,-1);
\draw [->] (3.5,1.25)--(4.3,1.25);
\draw (4.5,1.25) node {$q$};
\draw  (5,1)--(5,1.5)--(6,1.5)--(6,1)--(5,1);
\draw (5.5,1.25) node {$K_{n}$};
\draw (5.1,1.5)--(5.1,2.5);
\draw (5.9,1.5)--(5.9,2.5);
\draw (5.1,1)--(5.1,-1);
\draw (5.9,1)--(5.9,-1);
\draw (4.7,2.5)--(4.7,-1);
\draw (6.5,2)..controls (6.7,1) and (6.7,1)..(6.5,0);
\end{tikzpicture}
\end{center}
where throughout the notes we use the change of variables
\begin{equation}
\label{eq: change of variables}
    q=Q^2,\quad t=T^2Q^{-2},\quad a=AQ^{-2}
\end{equation}.

These pictures should be read in the following way: each picture
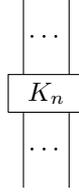
\begin{center}
\begin{tikzpicture}
\draw  (0,1)--(0,1.5)--(1,1.5)--(1,1)--(0,1);
\draw (0.2,0)--(0.2,1);
\draw (0.8,0)--(0.8,1);
\draw (0.5,0.5) node {$\cdots$};
\draw (0.5,2) node {$\cdots$};
\draw (0.5,1.25) node {$K_{n}$};
\draw (0.2,1.5)--(0.2,2.5);
\draw (0.8,1.5)--(0.8,2.5);
\end{tikzpicture}
\end{center}
corresponds to a (bounded) complex of Soergel bimodules, and stacking crossings on top (resp. bottom) of this picture means tensoring by the corresponding Rouquier complex on the right (resp. left). The closing of the last strand in (3) denotes taking $\HH_{\C[x_{n+1}]}(K_{n+1})$, which yields a complex of Soergel bimodules in one variable less, identified with a shifted sum of the complexes $K_n$.

\bigskip

\begin{example}
Let us check that the complex  $K_2=[R(-4)\buildrel{\Delta}\over\to B(-3)\buildrel{x_{i} - x'_{i}} \over \to B(-1)\buildrel m \over\to R]$ satisfies the above properties, in particular, eats crossings. 
Indeed, $K_2=[T^{-1}(-4)\to T]$, so thanks to Lemma \ref{lem: tt-1} and Example \ref{ex: two strand} we have
$$
TK_2=[R(-4)\to T^2]=[R(-4)\to B(-3)\to B(-1)\to R] = K_2.
$$
One can check that the differentials agree with this computation. 

Alternatively, one can check that $K_2$, as a complex of vector spaces (or left $R$-modules) is acyclic. If we tensor it with $B$, we get an acyclic complex of free $B$-modules, so it must be contractible. Therefore $BK_2\simeq 0$ and $TK_2\simeq K_2$.

\begin{example}\label{ex: t22}
Continuing with Example \ref{ex: two strand homology}, the Poincar\'e polynomials of $\HHH^{0}(T(2,2))$, $\HHH^{0}(T(2,3))$ in the $q, t$-variables become
$$
\HHH^{0}(T(2,2)) = \frac{qt^{-1}}{(1 - q)^2} + \frac{1}{1-q}=t^{-1}\frac{q+t-qt}{(1-q)^2}, \quad \HHH^{0}(T(2,3)) = \frac{1 + qt^{-1}}{1 - q}=t^{-1}\frac{q+t}{1-q}.
$$
In general, one can check that for two-stranded torus \emph{knots}:
$$
\HHH^{0}(T(2,2k+1)) = \frac{1 + qt^{-1} + q^{2}t^{-2} + \cdots + q^{k}t^{-k}}{1 - q}.
$$
Note that (up to a power of $t$ which corresponds to an overall grading shift) the Poincar\'e polynomials are given by rational functions where the numerator is symmetric in $q$ and $t$, and the denominator is a power of $(1-q)$. The symmetry between $q$ and $t$ is much less clear in variables $Q$ and $T$, which is one of the motivations behind the change of variables \eqref{eq: change of variables}. See also Theorem \ref{th: tautological} below.
\end{example}

%If we want to work with \emph{graded} bimodules, however, the reader can check that the complex $[T^{-1} \to T] = [R \to B(1) \to B(-1) \to R]$ does \emph{not} eat crossings. However, in this case one can check that the following complex in $\CK^{-}(\SBim)$ works:
%$$
%K_2 = [\dots \to B(-5) \to B(-3) \to B(-1) \to R]
%$$
%\noindent where the differentials alternate as in Example \ref{ex: two strand}. It is easy from our computations there that, indeed, $TK_2 \cong T$. Alternatively, the complex $K_2$ is acylic.

%There is in fact a representation-theoretic reason why, when we take gradings into account, there \emph{cannot} exist a complex $K_n \in \CK^b(\SBim_n)$ with the properties that we want, and it is necessary to pass to the unbounded homotopy category $\CK^{-}(\SBim_n)$. In a nutshell, $K_n$ is a \emph{categorification} of the trivial idempotent of the \emph{Hecke algebra}. We will not go into this, but we remark that the definition of such an idempotent involves dividing by powers of $(1-Q)$, where $Q$ is the grading shift. Indeed, Hogancamp's constructions in \cite{Hog} are in $\CK^{-}(\SBim)$. We will ignore these technicalities. 

\end{example}

In the following section we show how to use the properties of $K_n$ to recursively compute the link homology. Note that any combination of $a,q,t$ has even (homological) $T$-degree. The recursion would follow from the repeated simplification of the braid diagram using the above relations and the following standard lemma:

\begin{lemma}
Suppose that we have an exact triple of complexes 
$$
0\to A_{\bullet}\to B_{\bullet}\to C_{\bullet}\to 0
$$
and the homology of both $A_{\bullet}$ and $C_{\bullet}$ is supported in even homological degrees. Then the homology of $B_{\bullet}$ is supported in even homological degrees and 
$$
H_k(B_{\bullet})=H_k(A_{\bullet})\oplus H_k(C_{\bullet})\ \text{for all}\ k.
$$
\end{lemma}

For the reader's convenience, we also write an explicit recursion from \cite{HM}. By a {\bf binary sequence} we mean a (possibly empty) sequence of $0$'s and $1$'s. If $v = v_1\dots v_{\ell}$ is a binary sequence, we denote $|v| := \sum v_{i}$, the number of $1$'s appearing in $v$. 

\begin{theorem}[\cite{HM}]
\label{th: torus knot recursions}
Let $v$ and $w$ be a pair of binary sequences with $|v| = |w|$.  Let $p(v, w)$ denote the
unique family of polynomials, indexed by such pairs of binary sequences, satisfying
\begin{itemize}
    \item [(1)] 
    $p(\emptyset, 0^n) = \left(\frac{1+a}{1-q}\right)^n$
and $p(0^m, \emptyset) =\left(\frac{1+a}{1-q}\right)^m$
 \item [(2)] $p(v1, w1) = (t^{\ell} + a)p(v, w)$, where $|v| = |w| =\ell$.
\item [(3)] $p(v0, w1) = p(v, 1w).$
\item[(4)] $p(v1, w0) = p(1v, w).$
\item[(5)] $p(v0,w0) = t^{-\ell}p(1v,1w)+ qt^{-\ell}p(0v,0w)$, where $|v| = |w| = \ell$.
 \end{itemize}
Then the triply graded Khovanov-Rozansky homology of $T(m, n)$ is free over $\Z$ of graded rank
$p(0^m, 0^n)$.
\end{theorem}

\subsection{Examples}

\begin{example}
\label{ex: two strand recursion}
Let us use the recursion to compute the homology of two-strand torus links. We can write

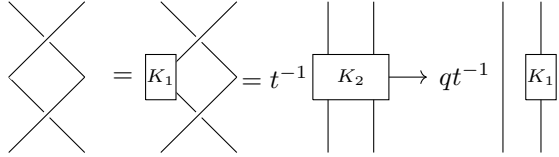
\begin{center}
\begin{tikzpicture}
\draw (1,0)--(0,1);
\draw [line width=3, white] (0,0)--(1,1);
\draw (0,0)--(1,1);
\draw (1,1)--(0,2);
\draw [line width=3, white] (0,1)--(1,2);
\draw (0,1)--(1,2);
\draw (1.5,1) node {$=$};
\draw (3,0)--(2,1);
\draw [line width=3, white] (2,0)--(3,1);
\draw (2,0)--(3,1);
\draw (3,1)--(2,2);
\draw [line width=3, white] (2,1)--(3,2);
\draw (2,1)--(3,2);
\filldraw [white] (1.8,0.7)--(1.8,1.3)--(2.2,1.3)--(2.2,0.7)--(1.2,0.7); 
\draw (1.8,0.7)--(1.8,1.3)--(2.2,1.3)--(2.2,0.7)--(1.8,0.7); 
\draw (2,1) node {\scriptsize $K_1$};
\draw (3.5,1) node {$=t^{-1}$};
\draw (4,0.7)--(4,1.3)--(5,1.3)--(5,0.7)--(4,0.7);
\draw (4.5,1) node {\scriptsize $K_2$};
\draw (4.2,0)--(4.2,0.7);
\draw (4.2,1.3)--(4.2,2);
\draw (4.8,0)--(4.8,0.7);
\draw (4.8,1.3)--(4.8,2);
\draw [->] (5,1)--(5.5,1);
\draw (6,1) node {$qt^{-1}$};
\draw (6.5,0)--(6.5,2);
\draw (7,0)--(7,0.7);
\draw (7,1.3)--(7,2);
\draw (6.8,0.7)--(6.8,1.3)--(7.2,1.3)--(7.2,0.7)--(6.8,0.7);
\draw (7,1) node {\scriptsize $K_1$};
\end{tikzpicture}
\end{center}
The first term evaluates to $t^{-1}(t+a)(1+a)$ while the second term evaluates to $t^{-1}q(1+a)^2$. Since both of them have even homological shifts, the differential vanishes, and the Poincar\'e polynomial equals
$$
\frac{t^{-1}(t+a)(1+a)}{1-q}+\frac{t^{-1}q(1+a)^2}{(1-q)^2}=\frac{t^{-1}(1+a)}{(1-q)^2}(t+a-qt-aq+q+aq)=
$$
$$
\frac{t^{-1}(1+a)}{(1-q)^2}(t+q-qt+a)=\frac{t^{-1}(1+a)}{(1-q)}(t+\frac{q}{1-q}+\frac{a}{1-q}).
$$
We can apply the same relation with $T(2,m)$ braid added below. On the left we get $T(2,m+2)$ while on the right we get $T(2,m)$ followed by $K_2$ (which is the same as $K_2$ since it eats crossings) and $T(2,m)$. By induction in $m$, we conclude that for all positive $m$ the homology of $T(2,m)$ is supported in even homological degrees, and the Poincar\'e polynomial is given by 
$$
\HHH(T(2,m+2))=\frac{t^{-1}(t+a)(1+a)}{1-q}+qt^{-1}
\HHH(T(2,m))
$$
This recursion is easy to solve with the initial conditions
$$
\HHH(T(2,0))=\frac{(1+a)^2}{(1-q)^2},\ \HHH(T(2,1))=\frac{(1+a)}{(1-q)}.
$$
One can compare this with Example \ref{ex: two strand homology}.
\end{example}

\begin{example}
For a more complicated example, let us compute the homology of the $(3,3)$ torus link. Consider the braid

\begin{center}
\begin{tikzpicture}
\draw (-1,1) node {$L_3=$};
\draw (1,0)--(0,1);
\draw [line width=3,white] (0,0)--(0.5,1);
\draw (0,0)--(0.5,1)--(0,2);
\draw [line width=3,white] (0.5,0)--(1,1);
\draw (0.5,0)--(1,1)--(0.5,2);
\draw [line width=3,white] (0,1)--(1,2);
\draw (0,1)--(1,2);
\end{tikzpicture}
\end{center}

It is easy to see that $T(3,3)=T(2,2)\cdot L_3$. By applying the recursion from Example \ref{ex: two strand recursion} to $T(2,2)$, we resolve $T(3,3)$ by $t^{-1}K_2L_3$ and $t^{-1}qL_3$. By the main recursion, $K_2L_3$ is resolved by $t^{-2}K_3$ and $t^{-2}q(K_2\sqcup 1)$. To resolve $L_3$, we can write

\begin{center}
\begin{tikzpicture}
\draw (-1,1) node {$L_3=$};
\draw (1,0)--(0,1);
\draw [line width=3,white] (0,0)--(0.5,1);
\draw (0,0)--(0.5,1)--(0,2);
\draw [line width=3,white] (0.5,0)--(1,1);
\draw (0.5,0)--(1,1)--(0.5,2);
\draw [line width=3,white] (0,1)--(1,2);
\draw (0,1)--(1,2);
\draw (1.5,1) node {$=t^{-1}$};
\draw (2,0.7)--(2,1.3)--(3,1.3)--(3,0.7)--(2,0.7);
\draw (2.5,1) node {\scriptsize $K_2$};
\draw (2.2,0)--(2.2,0.7);
\draw (2.2,1.3)--(2.2,2);
\draw (3.2,0)--(2.8,0.7);
\draw [line width=3,white] (2.8,0)--(3.2,0.7);
\draw (2.8,0)--(3.2,0.7);
\draw (3.2,1.3)--(2.8,2);
\draw [line width=3,white] (2.8,1.3)--(3.2,2);
\draw (2.8,1.3)--(3.2,2);
\draw (3.2,0.7)--(3.2,1.3);
\draw [->] (3.5,1)--(4,1);
\draw (4.5,1) node {$qt^{-1}$};
\draw (5,0)--(5,2);
\draw (6,0)--(5.5,1);
\draw [line width=3,white] (5.5,0)--(6,1);
\draw (5.5,0)--(6,1);
\draw (6,1)--(5.5,2);
\draw [line width=3,white] (5.5,1)--(6,2);
\draw (5.5,1)--(6,2);
\end{tikzpicture}
\end{center}
Finally, we can use the recursion to write
\begin{center}
\begin{tikzpicture}
\draw (2,0.7)--(2,1.3)--(3,1.3)--(3,0.7)--(2,0.7);
\draw (2.5,1) node {\scriptsize $K_2$};
\draw (2.2,0)--(2.2,0.7);
\draw (2.2,1.3)--(2.2,2);
\draw (3.2,0)--(2.8,0.7);
\draw [line width=3,white] (2.8,0)--(3.2,0.7);
\draw (2.8,0)--(3.2,0.7);
\draw (3.2,1.3)--(2.8,2);
\draw [line width=3,white] (2.8,1.3)--(3.2,2);
\draw (2.8,1.3)--(3.2,2);
\draw (3.2,0.7)--(3.2,1.3);
\draw (2.2,0)..controls (2.2,-0.5) and (1.8,-0.5)..(1.8,0);
\draw (2.2,2)..controls (2.2,2.5) and (1.8,2.5)..(1.8,2);
\draw (1.8,0)--(1.8,2);
\draw (4,1) node {$=(t+a)$};
\draw (6,0)--(5,1);
\draw [line width=3, white] (5,0)--(6,1);
\draw (5,0)--(6,1);
\draw (6,1)--(5,2);
\draw [line width=3, white] (5,1)--(6,2);
\draw (5,1)--(6,2);
\filldraw [white] (4.8,0.7)--(4.8,1.3)--(5.2,1.3)--(5.2,0.7)--(5.2,0.7); 
\draw (4.8,0.7)--(4.8,1.3)--(5.2,1.3)--(5.2,0.7)--(4.8,0.7); 
\draw (5,1) node {\scriptsize $K_1$};
\end{tikzpicture}
\end{center}
which we already computed. As a result, we proved that $\HHH(T(3,3))$ is supported in even homological degrees, and refer the reader to \cite{EH,HM} for the final answer. 
\end{example}

\begin{example}
\label{ex: T34 homology}
A similar computation yields the following Khovanov-Rozansky homology of $T(3,4)$ torus knot, first computed in \cite{DGR}:
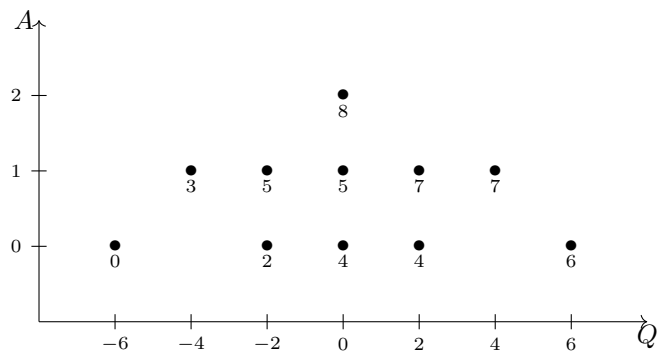
\begin{center}
    \begin{tikzpicture}
        \draw (0,0) node {$\bullet$};
        \draw (2,0) node {$\bullet$};
        \draw (3,0) node {$\bullet$};
        \draw (4,0) node {$\bullet$};
        \draw (6,0) node {$\bullet$};
        \draw (1,1) node {$\bullet$};
        \draw (2,1) node {$\bullet$};
        \draw (3,1) node {$\bullet$};
        \draw (4,1) node {$\bullet$};
        \draw (5,1) node {$\bullet$};
        \draw (3,2) node {$\bullet$};
        \draw[->] (-1,-1)--(7,-1);
        \draw [->] (-1,-1)--(-1,3);
        \draw (7,-1.2) node {$Q$};
        \draw (-1.2,3) node {$A$};
        \draw (0,-0.2) node {\scriptsize $0$};
        \draw (2,-0.2) node {\scriptsize $2$};
        \draw (3,-0.2) node {\scriptsize $4$};
        \draw (4,-0.2) node {\scriptsize $4$};
        \draw (6,-0.2) node {\scriptsize $6$};
        \draw (1,0.8) node {\scriptsize $3$};
        \draw (2,0.8) node {\scriptsize $5$};
        \draw (3,0.8) node {\scriptsize $5$};
        \draw (4,0.8) node {\scriptsize $7$};
        \draw (5,0.8) node {\scriptsize $7$};
        \draw (3,1.8) node {\scriptsize $8$};
        \draw (0,-0.9)--(0,-1.1);
        \draw (1,-0.9)--(1,-1.1);
        \draw (2,-0.9)--(2,-1.1);
        \draw (3,-0.9)--(3,-1.1);
        \draw (4,-0.9)--(4,-1.1);
        \draw (5,-0.9)--(5,-1.1);
        \draw (6,-0.9)--(6,-1.1);
        \draw (0,-1.3) node {\scriptsize $-6$};
        \draw (1,-1.3) node {\scriptsize $-4$};
        \draw (2,-1.3) node {\scriptsize $-2$};
        \draw (3,-1.3) node {\scriptsize $0$};
        \draw (4,-1.3) node {\scriptsize $2$};
        \draw (5,-1.3) node {\scriptsize $4$};
        \draw (6,-1.3) node {\scriptsize $6$};
        \draw (-1.1,0)--(-0.9,0);
        \draw (-1.1,1)--(-0.9,1);
        \draw (-1.1,2)--(-0.9,2);
        \draw (-1.3,0) node {\scriptsize $0$};
        \draw (-1.3,1) node {\scriptsize $1$};
        \draw (-1.3,2) node {\scriptsize $2$};
    \end{tikzpicture}
\end{center}
The homology is 11-dimensional (the generators correspond to the dots in the picture) and concentrated in three $A$-degrees.
The $Q$-degree is marked on a horizontal axis, $A$-degree on the vertical axis and $T$-degree is marked next to the dots.
%\fixme{Check normalization!!!}
\end{example}

\section{Braid varieties}

\subsection{Braid varieties and their properties}

In this lecture we describe a geometric model for Khovanov-Rozansky homology of positive braids. Given $1\le i\le n-1$, we consider the matrix
$$
B_i(z)=\left(\begin{matrix}
1     &   0   &  & \cdots &       & 0 \\
0      &\ddots & &        &       & \\
\vdots &      &0 & 1        &       & \vdots\\
       &    &1 & z        &       &  \\
       &    &  &        & \ddots & 0 \\
 0      &    &\cdots   &       &   0     & 1\\
\end{matrix}
\right)
$$
where the nontrivial $2\times 2$ block is at the $i$-th and $(i+1)$-st positions.

\begin{lemma}
\label{lem: matrix braid rel}
$B_i(z_1)B_{i+1}(z_2)B_i(z_3)=B_{i+1}(z_3)B_i(z_2-z_1z_3)B_{i+1}(z_1)$.
\end{lemma}

\begin{proof}
Without loss of generality we may work with $3 \times 3$-matrices, $i=1$. The left hand side is
$$
\begin{pmatrix}
 0 & 1 & 0 \\
 1 & z_1 & 0\\
 0 & 0 & 1
\end{pmatrix}
\begin{pmatrix}
 1 & 0 & 0 \\
 0 & 0 & 1\\
 0 & 1 & z_2
\end{pmatrix}
\begin{pmatrix}
 0 & 1 & 0 \\
 1 & z_3 & 0\\
 0 & 0 & 1
\end{pmatrix}=
\begin{pmatrix}
  0 & 0 & 1 \\
 1 & 0 & z_1 \\
 0 & 1 & z_2 
\end{pmatrix}\begin{pmatrix}
 0 & 1 & 0 \\
 1 & z_3 & 0\\
 0 & 0 & 1
\end{pmatrix}
=\begin{pmatrix}
  0 & 0 & 1 \\
 0 & 1 & z_1 \\
 1 & z_3 & z_2 
\end{pmatrix}
$$
and the right hand side is 
$$
\begin{pmatrix}
 1 & 0 & 0 \\
 0 & 0 & 1\\
 0 & 1 & z_3
\end{pmatrix}
\begin{pmatrix}
 0 & 1 & 0 \\
 1 & z_2-z_1z_3 & 0\\
 0 & 0 & 1
\end{pmatrix}
\begin{pmatrix}
 1 & 0 & 0 \\
 0 & 0 & 1\\
 0 & 1 & z_1
\end{pmatrix}
=\begin{pmatrix}
  0 & 1 & 0 \\
 0 & 0 & 1 \\
 1 & z_2-z_1 z_3 & z_3 \\
\end{pmatrix}
\begin{pmatrix}
 1 & 0 & 0 \\
 0 & 0 & 1\\
 0 & 1 & z_1
\end{pmatrix}=
\begin{pmatrix}
  0 & 0 & 1 \\
 0 & 1 & z_1 \\
 1 & z_3 & z_2 
\end{pmatrix}.
$$
\end{proof}

Given a {\bf positive} braid
$$
\beta=\sigma_{i_1}\cdots \sigma_{i_r}
$$
we define, following Mellit \cite{Mellit2}, the {\bf braid matrix}
$$
B_{\beta}(z_1,\ldots,z_r)=B_{i_1}(z_1)\cdots B_{i_r}(z_r).
$$
and the {\bf braid variety}
$$
X(\beta)=\left\{(z_1,\ldots,z_r):B_{\beta}(z_1,\ldots,z_r)\ \text{is upper-triangular}\right\}\subset \C^r.
$$
It follows from Lemma \ref{lem: matrix braid rel} and the obvious equality $B_{i}(z)B_{j}(w) = B_{j}(w)B_{i}(z)$ for $|i - j| > 1$, that $B_{\beta}$ satisfies the braid relations \eqref{eq: braid} up to a change of variables, and the variety $X({\beta})$ does not depend on the braid word for $\beta$ up to isomorphism. Clearly, $X(\beta)$ is an affine algebraic variety in $\C^r$ cut out by $\binom{n}{2}$ equations.

\begin{example}\label{ex: xsigma3}
To describe  $X(\sigma^3)$ we compute
\begin{align*}
\begin{pmatrix}
0 & 1 \\
1 & z_1
\end{pmatrix}
\begin{pmatrix}
0 & 1 \\
1 & z_2
\end{pmatrix}
\begin{pmatrix}
0 & 1 \\
1 & z_3
\end{pmatrix}
&=\begin{pmatrix}
1 & z_2 \\
z_1 & 1 + z_1 z_2
\end{pmatrix}
\begin{pmatrix}
0 & 1 \\
1 & z_3
\end{pmatrix} \\
&= \begin{pmatrix}
z_2 & 1+ z_2z_3 \\
1+z_1z_2 & z_1+z_3+z_1z_2z_3
\end{pmatrix}
\end{align*}
This 
is upper triangular if and only if  $1+z_1z_2=0$, equivalently $z_1\neq 0$, so $X(\sigma^3)=\C^*_{z_1}\times \C_{z_3}$.
\end{example}

\begin{example}\label{ex: sigma4}
Similarly, $X(\sigma^4)$ will be related to the matrix 
\begin{align*}
    \begin{pmatrix}
z_2 & 1+ z_2z_3 \\
1+z_1z_2 & z_1+z_3+z_1z_2z_3
\end{pmatrix}\begin{pmatrix}
0 & 1 \\
1 & z_4
\end{pmatrix}&=\\
\begin{pmatrix}
 1+ z_2z_3 & z_2+z_4+z_2z_3z_4 \\
 z_1+z_3+z_1z_2z_3 & 1+z_1z_2+z_4z_1+z_3z_4+z_1z_2z_3z_4
\end{pmatrix}
\end{align*}
This is upper-triangular if $z_1+z_3+z_1z_2z_3=0$. This is a hypersurface in $\C^3$ times $\C_{z_4}$. Note that we can rewrite this equation as $z_1+z_3(1+z_1z_2)=0$. If $1+z_1z_2=0$ then $z_1=0$, contradiction. Therefore $1+z_1z_2\neq 0$ and $z_3=-\frac{z_3}{1+z_1z_2}$, so that 
$$
X(\sigma^4)=\{1+z_1z_2\neq 0\}\times \C_{z_4}
$$
and hence it is smooth of dimension 3.
\end{example}

\begin{example}\label{ex: xsigma5}
Similarly, $X(\sigma^5)$ corresponds 
to $$\begin{pmatrix}
  z_2+z_4+z_2z_3z_4 & z_2 z_3+z_2z_5+z_2 z_3z_4z_5+ z_4 z_5+1 \\
 1+z_1z_2+z_4z_1+z_3z_4+z_1z_2z_3z_4 & z_2 z_3 z_1+z_2 z_5 z_1+z_2 z_3 z_4 z_5 z_1+z_4 z_5 z_1+z_1+z_3+z_3 z_4 z_5+z_5
\end{pmatrix}$$
This is upper-triangular if 
$$
1+z_1z_2+z_4z_1+z_3z_4+z_1z_2z_3z_4=(1+z_1z_2)+z_4(z_1+z_3+z_1z_2z_3)=0.
$$
Similarly to the previous case, one can check that this hypersurface is isomorphic to the open subset $\{z_1+z_3+z_1z_2z_3\neq 0\}$ and does not depend on $z_5$, so that it is smooth of dimension 4.
\end{example}

Note that in all these examples there is a torus action on $\X(\sigma^k)$ defined by the equation 
$$t.(z_1,z_2,z_3,z_4,z_5)\mapsto (tz_1,t^{-1}z_2,tz_3,t^{-1}z_4,tz_5)$$
This can be generalized as follows.

\begin{lemma}
\label{lem: torus action}
(a) We have $\diag(t_1,\ldots,t_n)\cdot B_i(z)=B_i(z')D$ for some $z'$ and some diagonal matrix $D$. 

\noindent (b) Part (a) defines the action of the torus $T:=(\C^*)^{n-1}$ on the braid variety $X(\beta)$ for any $n$-strand braid $\beta$. 
 
\noindent (c) The torus action on $X(\beta)$ is free if  $\beta$ closes up to a knot (with one component). In this case, we have $$X(\beta)=\left[X(\beta)/T\right]\times T.$$

\end{lemma}
\begin{proof}
\begin{enumerate}[(a)] 
\item It is clearly enough to verify for a single crossing, and here it is enough to work with $2 \times 2$-matrices. We have:
$$\begin{pmatrix}
 t_1 & 0 \\ 0 & t_2
\end{pmatrix}\begin{pmatrix}
 0 & 1 \\
  1 & z 
\end{pmatrix}
=\begin{pmatrix}
 0 & t_1 \\ t_2 & t_2 z
\end{pmatrix}=\begin{pmatrix}
0 & 1 \\
1 & \frac{t_2}{t_1} z
\end{pmatrix}\begin{pmatrix}
 t_2  & 0 \\
 0 & t_1
\end{pmatrix}$$
\item $t\in (\C^*)^n$ acts on the left - moving everything to the right in $B_\beta(\vec{z})$ using (a) multiplies entries of  $\vec{z}$ by $t_{w_k(i_{k+1})}t_{w_k(i_k)}^{-1}$ and permutes entries of $t$ by $w(\beta)$. The upper-triangularity condition doesn't change. So we get an action.
\item If the closure is a knot, the above dictates $t_i t_{w(i)}^{-1}=1$ for all $i$. So the stabilizers are all trivial. Now assume that $\beta$ closes up to a knot, and let $r:= \ell(\beta)$ be the length of the braid $\beta$. If we consider the subvariety $X'(\beta)$ consisting of $z \in X(\beta)$ such that $B_{\beta}(z)$ has a fixed diagonal $((-1)^{r}, 1, \dots, 1)$ then it is straightforward to see that $X'(\beta) \cong X(\beta)/T$ and that $X(\beta) \cong X'(\beta) \times T$. 
\end{enumerate}
\end{proof}

The following theorem generalizes Examples \ref{ex: xsigma3}--\ref{ex: xsigma5} to an arbitrary number of strands. 

\begin{theorem}[\cite{CGGS,CGGS2}]
\label{th: braid variety}
Suppose that $\beta=\gamma\Delta$ where $\Delta$ is the positive half-twist braid and $\gamma$ is an arbitrary positive braid. Then the following holds:

(a) $X(\beta)$ is non-empty if and only if $\gamma$ contains $\Delta$ as a subword. In what follows we assume that this is the case.

(b) $X(\beta)$ is smooth of (expected) dimension $\ell(\beta)-\binom{n}{2}=\ell(\gamma)$.

(c) The variety $X(\beta)$ is an invariant of $\gamma\Delta^{-1}$ under conjugation and {\bf positive} stabilization.

(d) $X(\beta)$ has a smooth compactification (depending on a braid word for $\beta$) where the complement to $X(\beta)$ is a normal crossings divisor with stratification labeled by the subwords of $\gamma$ containing $\Delta$.
\end{theorem}

\begin{remark}
Note that the braid $\gamma\Delta^{-1}=\beta\Delta^{-2}$ which appears in part (c) is not necessary positive. 
\end{remark}

\begin{remark}
The combinatorics of the strata in the compactification in part (d) is described by the {\bf subword complex} of $\gamma$ defined by Knutson and Miller \cite{KM}. In fact, this compactification coincides with the {\bf brick variety} of the braid word $\beta$, studied by Escobar in \cite{escobar}.%, see in particular Theorem 3.6 in {\it loc. cit.}. 
\end{remark}

It is sometimes more convenient to work with the following modification of the braid variety
\[
X(\beta;w_0) := \{(z_1, \dots, z_r) : B_{\beta}(z_1, \dots, z_r)w_0 \; \text{is upper-triangular}\} \subseteq \C^r 
\]
\noindent where $w_0 \in S_n$ is the longest element, $w_0 = [n, n-1, \dots, 1]$ which is interpreted in the equation above as a permutation matrix. In fact, in the context of Theorem \ref{th: braid variety} we have, see \cite{CGGS}
$$
X(\beta) = X(\gamma\Delta) = X(\gamma;w_0) \times \C^{\binom{n}{2}}.
$$

\noindent Note that this explains the appearance of the $\C$-factor in Examples \ref{ex: xsigma3}--\ref{ex: xsigma5}.

\begin{example}
As in Example \ref{ex: sigma4}, we get 
$$
X(\sigma^4)=\{1+z_1z_2\neq 0\}\times \C_{z_4}=X(\sigma^3;w_0)\times \C
$$
where
$$
X(\sigma^3;w_0)=\{1+z_1z_2\neq 0\}.
$$
This compactifies to $\P^1\times \P^1$, with the following strata of the complement: the hyperbola $\{1+z_1z_2=0\}$, two lines at infinity, and three pairwise intersection points of these. The strata are in bijection with nonempty subwords of $\sigma^3$.
\end{example}

Let us briefly comment on the ideas behind the proof of Theorem \ref{th: braid variety}(c), following \cite{CGGS2}. First, to the braid $\beta=\gamma\Delta$ one can associate a {\bf Legendrian link}  which has the smooth type of the closure of $\beta\Delta^{-2}=\gamma\Delta^{-1}$. This is done using ``pigtail closure'' of $\beta$, see \cite{CN,CGGS2}. Next, to any Legendrian link Chekanov \cite{Che} associated a dg algebra $\CCC$ and proved that its cohomology is a Legendrian link invariant. See also \cite{CN} for a construction of $\CCC$ over the integers and more details.

An important invariant of a dga $\CCC$ is its {\bf augmentation variety} defined as $\Aug(\CCC)=\Spec H^0(\CCC)$. By the work of K\'alm\'an \cite{Kalman} for a positive braid $\gamma$ the dga $\CCC$ coincides with the Koszul complex for the equations defining $X(\gamma;w_0)$, so that $\Aug(\CCC)=X(\gamma;w_0)$. If $\gamma$ is not positive but is equivalent to a positive braid, one can still describe $\Aug(\CCC)$ as the quotient of some explicit algebraic variety by a free action of some affine space $\C^w$ with coordinates corresponding to the negative crossings in $\gamma$. 

See \cite{CGGS2} for an explicit description of $\CCC$, $\Aug(\CCC)$ and their behaviour under braid moves, conjugation and positive stabilization.

\subsection{Homology of braid varieties}

The following result has been discussed in e.g. \cite{Mellit2,STZ}, but in this form it was stated and proved only recently in \cite[Corollary 4]{Trinh}. 

\begin{theorem}
\label{th: braid variety homology}
The $T=(\C^*)^{n}$-equivariant Borel-Moore homology of $X(\beta)$ has a nontrivial weight filtration. The associated graded for this filtration is isomorphic to 
$$
\gr_{W}H^T_{*,BM}(X(\beta))\simeq \HHH^n(\beta).
$$
Here the weight grading corresponds, up to an overall shift in {\em loc. cit.}, to the $q$-grading, and the homological grading corresponds to the sum of the $q$ and $t$-gradings. The action of the variables $x_i$ on the right hand side corresponds to the equivariant parameters, that is, the generators of the $T$-equivariant homology of a point.
\end{theorem}
\begin{remark}
Note that the variety used in \cite[Corollary 4]{Trinh} slightly differs from our $X(\beta)$ and loc. cit. uses $G$-equivariant homology instead of $T$-equivariant one. Nevertheless, by \cite[Remark B.4.3]{Trinh} the corresponding homologies are isomorphic.
\end{remark}
\begin{remark}
Similar results in sheaf- and Soergel-theoretic form appear in \cite{WW, WW2, Rouquier}.
\end{remark}
\begin{remark}
Note that by the main result of \cite{GHMN} we have
$$
\HHH^n(\beta)=\HHH^0(\beta\Delta^{-2})=\HHH^0(\gamma\Delta^{-1}).
$$
As we discussed above, this is invariant under conjugation and positive stabilization of $\gamma\Delta^{-1}$ in agreement with Theorem \ref{th: braid variety}(c).
\end{remark}
\begin{remark}
In \cite{Trinh} Trinh  has also extended  Theorem \ref{th: braid variety homology} to all Hochschild degrees using Springer theory. 
\end{remark}
\begin{remark}
Since $X(\beta)$ is smooth, on the level of triply graded vector spaces Theorem \ref{th: braid variety homology} may be also stated in terms of usual singular cohomology. Namely, the Verdier dualizing sheaf of $X(\beta)$ is simply $\omega_{X(\beta)}\cong \C[-\dim X(\beta)]$, and the equivariant BM homology agrees as a doubly graded vector space with the equivariant cohomology. This is however not the case  for the more general varieties appearing in the extension to higher Hochschild degrees, and equivariant BM homology (or up to a linear duality, cohomology with compact supports) seems to be the most natural choice also in view of Section \ref{sec: gasf}. 
\end{remark}
Let us sketch some ideas in the proof of Theorem \ref{th: braid variety homology}, referring the reader to \cite{Trinh} for more details.

First, recall the bimodules $B_i=R\otimes_{R^{s_i}}R$. To give a geometric interpretation to these, consider the {\bf Bott-Samelson variety}  \cite{Brion}
$$
\mathrm{BS}_i=\{(\CF,\CF'): \CF,\CF'\ \text{complete flags}, \CF_j=\CF'_j\ \text{for}\ j\neq i.\}
$$
We have line bundles $\CL_j=\CF_j/\CF_{j-1}$ and $\CL'_j=\CF'_j/\CF'_{j-1}$ and the corresponding Chern classes
$$
x_j=c_1(\CL_j),\ x'_j=c_1(\CL'_j).
$$
Clearly, $\CL_j\simeq \CL'_j$  and $x_j=x'_j$ for $j\neq i,i+1$. Furthermore, 
\begin{equation}
\label{eq: rank 2 quotient}
\CF_{j+1}/\CF_{j-1}\simeq \CF'_{j+1}/\CF'_{j-1}
\end{equation}
is filtered both by $\CL_i,\CL_{i+1}$ and by $\CL'_i,\CL'_{i+1}$, so
the elementary symmetric functions in $x_i,x_{i+1}$ and $x'_i,x'_{i+1}$ describe Chern classes of the rank two bundle \eqref{eq: rank 2 quotient} and hence agree with each other. These are precisely the defining equations for $B_i$.

Given a tensor product $B_{i_1}\otimes\cdots \otimes B_{i_r}$, we can define the more general Bott-Samelson variety as the space of sequences of flags $(\CF^{(1)},\ldots,\CF^{(r+1)})$ such that $(\CF^{(s)},\CF^{(s+1)})$ satisfy the conditions for $\mathrm{BS}_{i_s}$ for all $s$.

Next, we associate some geometric objects to Rouquier complexes $T_i$ associated to {\bf positive} braids. We say that two flags $\CF,\CF'$ are in position $s_i$ if $\CF_j=\CF'_j$ for $j\neq i$ and $\CF_i\neq \CF'_i$. This is an open subset of $\mathrm{BS}_i$ complementary to the diagonal. To a Rouquier complex $T_{\beta}=T_{i_1}\otimes \cdots \otimes T_{i_r}$ we associate the {\bf open Bott-Samelson variety} $\mathrm{OBS}_{\beta}$ consisting of sequences of flags $(\CF^{(1)},\ldots,\CF^{(r+1)})$ such that $(\CF^{(t)},\CF^{(t+1)})$ are in position $s_{i_t}$. The open Bott-Samelson variety was considered in the work of Brou\'e, Michel and Deligne \cite{BM,Deligne}, and later in \cite{STZ}, it plays a prominent role in the Deligne-Lusztig theory \cite{DL}. In particular, Deligne proved in \cite{Deligne} that $\mathrm{OBS}_{\beta}$ is an invariant of the braid $\beta$ and does not change under braid relations up to a (canonical) isomorphism. 

To compute the homology of $\mathrm{OBS}_{\beta}$, one can use the exact sequence in equivariant cohomology coming from excision and inclusion-exclusion. Namely, the open Bott-Samelson variety has a natural embedding into the closed Bott-Samelson variety for the same braid word. Its complement is a union of closed Bott-Samelson varieties corresponding to subwords of $\beta$ with one letter skipped, the intersections of which correspond to subwords of $\beta$ with two letters skipped and so on. As a result, we can compute the homology of $\mathrm{OBS}_{\beta}$ using the homologies of closed Bott-Samelson varieties for all possible subwords of $\beta$. This is parallel to the expansion of $T_{\beta}$ as a complex built out of tensor products of $B_i$ for all subwords of $\beta$.

Finally, we need to compare the open Bott-Samelson variety $\mathrm{OBS}_{\beta}$ to the braid variety $X(\beta)$. This is given by the following 

%\begin{lemma}[\cite{CGGS2}]
%(a) The variety $X(\beta;w_0)$ is a subset of $\mathrm{OBS}_{\beta}$ where the first flag $\CF^{(1)}$ is chosen to be standard, and the last flag $\CF^{(r+1)}$ is chosen to be antistandard: $\CF^{(r+1)}=w_0\CF^{(1)}$.

%(b) The variety $X(\beta)$  is a subset of $\mathrm{OBS}_{\beta}$ where both the first and the last flags are chosen to be standard.
%\end{lemma}

%Identification of the first and last flag in (b) corresponds to closing the braid and computing $\HHH^n$.

\begin{lemma}[\cite{CGGS2}]
(a) The variety $X(\reflectbox{$\beta$};w_0)$ is a subset of $\mathrm{OBS}_{\beta}$ where the first flag $\CF^{(1)}$ is chosen to be standard, and the last flag $\CF^{(r+1)}$ is chosen to be antistandard: $\CF^{(r+1)}=w_0\CF^{(1)}$.

(b) The variety $X(\reflectbox{$\beta$})$  is a subset of $\mathrm{OBS}_{\beta}$ where both the first and the last flags are chosen to be standard.

Here, \reflectbox{$\beta$} is the braid $\beta$ read in the opposite direction.
\end{lemma}

Identification of the first and last flag in (b) corresponds to closing the braid and computing $\HHH^n$.

The compactification of $X(\beta)$ mentioned in Theorem \ref{th: braid variety} corresponds, up to replacing $\beta$ by \reflectbox{$\beta$}, to the compactification of $\mathrm{OBS}_{\beta}$ by $\mathrm{BS}_{\beta}$. Note that this compactification depends on the choice of a braid word for $\beta$.

\begin{example}
Let us compute the homology of the varieties $X(\beta)$ for $\beta=\sigma^3,\sigma^4$ from Examples \ref{ex: xsigma3} and \ref{ex: sigma4}. Note that for $\beta=\sigma^3$ %,\sigma^5$
by Lemma \ref{lem: torus action} the torus action is free, and the equivariant homology is related to the non-equivariant homology by a trivial factor.

The homology of $X(\sigma^3)=\C^*\times \C$ is clear. To compute the homology of $X(\sigma^4)=\{1+z_1z_2\neq 0\}\times \C$ we use the Alexander duality. The hyperbola $\{1+z_1z_2=0\}$ is isomorphic to $\C^*$ and has nontrivial $H^0$ and $H^1$, or, equivalently, nontrivial cohomology with compact support $H^{2,c}$ and $H^{1,c}$. Now
$$
\widetilde{H_i}(\{1+z_1z_2\neq 0\})=H^{4-1-i,c}(\{1+z_1z_2=0\})
$$
so we have nontrivial $H^1$ and $H^2$. Since it is connected, we also have $H^0$. To sum up, we have
$$
H^0(X(\sigma^4))=H^1(X(\sigma^4))=H^2(X(\sigma^4))=\Z,
$$
and the other homologies vanish.
\end{example}

\begin{example}
The computation for $X(\sigma^5)$ from Example \ref{ex: xsigma5} is similar: by the Alexander duality  we have
$$
\widetilde{H_i}
\left(\{z_1+z_3+z_1z_2z_3\neq 0\}\right)=H^{6-1-i,c}
\left(\{z_1+z_3+z_1z_2z_3=0\}\right)=H^{5-i,c}\left(\{1+z_1z_2\neq 0\}\right)
$$
while by the Poincar\'e duality we have
$$
H^{5-i,c}(\{1+z_1z_2\neq 0\})=H_{i-1}(\{1+z_1z_2\neq 0\}).
$$
Therefore, 
$$
H^0(X(\sigma^5))=H^1(X(\sigma^5))=H^2(X(\sigma^5))=H^3(X(\sigma^5))=\Z,
$$
and the other homologies vanish. Note that $\sigma^5$ closes up to a knot, so $T=\C^*$ action on $X(\sigma^5)$ is free and by Lemma \ref{lem: torus action}(c) we have
$$
H^*(X(\sigma^5))=H^*(\C^*)\times H^*_{T}X(\sigma^5),
$$
so that 
$$
H^0_{T}(X(\sigma^5))=H^2_{T}(X(\sigma^5))=\Z
$$
and other equivariant homology vanish. This agrees with the $A=0$ Khovanov-Rozansky homology of the trefoil $T(2,3)$, see Example \ref{ex: two strand homology}. 
\end{example}

\begin{example}
\label{ex: E6 homology}
Consider the three strand braid $\beta=(\sigma_1\sigma_2)^7$ which corresponds to the $(3,4)$ torus knot obtained as the closure of $\beta\Delta^{-2}$. It has a free action of the torus $T=(\C^*)^2$ and the quotient $X(\beta\Delta^{-1};w_0)/T$ is isomorphic to the $E_6$ cluster variety (see Example \ref{ex: E6} below). Its homology 
$$
\gr^p_{W}H^k(X(\beta\Delta^{-1};w_0)/T)=\gr^p_{W}H^k_T(X(\beta))
$$
with the weight filtration was computed by Lam and Speyer in \cite{LS} and is given by the following table:
\begin{center}
\begin{tabular}{|c|c|c|c|c|c|c|c|}
\hline
  & $H^0$ & $H^1$ & $H^2$ & $H^3$ & $H^4$ & $H^5$ & $H^6$ \\
\hline
     $k-p=0$ & 1 & 0 & 1 & 0 & 1 & 0 & 1 \\
\hline
     $k-p=1$ &   0 & 0 & 0 & 0 & 1 & 0 & 0 \\
\hline    
\end{tabular}
\end{center}
This matches the bottom row of the Khovanov-Rozansky homology of $T(3,4)$ in Example \ref{ex: T34 homology}.
\end{example}

\subsection{From braid varieties to positroid varieties}

Following the recent work of Galashin and Lam \cite{GL,GL2} and \cite{CGGS2}, we relate the braid varieties to {\bf positroid varieties} in the Grassmannian defined in \cite{KLS}.

Recall that the Grassmannian $\mathrm{Gr}(k,n)$ parametrizes $k$-dimensional subspaces in $n$-dimensional space, which can be presented as the row span of a $k\times n$ matrix of maximal rank. Let $v_1,\ldots,v_n$ be the columns of such matrix, extend these periodically by setting $v_{i+n}=v_i.$

The {\bf open positroid stratum} $\Pi_{k,n}$ is defined by the conditions that all the cyclically consecutive $k\times k$ minors 
$$
\Delta_{i,\ldots,i+k-1}=\det(v_i,\ldots,v_{i+k-1})
$$
are nonzero.

\begin{lemma}[\cite{CGGS2}]
\label{ex: positroid}
We have 
$$
\Pi_{k,n}=(\C^*)^{n-k}\times X(\beta_{k,n-k}\Delta;w_0)
$$
where $\beta_{k,n-k}=(\sigma_1\cdots\sigma_{k-1})^{n-k}$ is the $(k,n-k)$ torus braid on $k$ strands.
\end{lemma}

\begin{example}
Consider the open positroid variety $\Pi_{2,4}$.
Since $\Delta_{1,2}\neq 0$, we can use row operations to write the matrix as 
$$
\left(\begin{matrix}
1 & 0 & a & b\\
0 & 1 & c & d\\
\end{matrix}\right)
$$
Other conditions on minors imply 
$$
a\neq 0, d\neq 0, ad-bc\neq 0. 
$$
By denoting $x=b/a, y=-c/d$ we get $xy+1\neq 0$ which is the defining inequality of $X(\sigma^3;w_0)$.
\end{example}

\begin{example}
Similarly, for $\Pi_{2,5}$ we get the matrix
$$
\left(\begin{matrix}
1 & 0 & a & b & e\\
0 & 1 & c & d & f\\
\end{matrix}\right)
$$
such that $a\neq 0, f\neq 0, ad-bc\neq 0, bf-de\neq 0$. Without loss of generality, we can get rid of $(\C^*)^3$ and assume that $a=f=ad-bc=1$, then $d=1+bc$. Now
$$
bf-de=b-(1+bc)e\neq 0.
$$
If we denote $z_1=b,z_2=c,z_3=-e$, we get
$$
z_1+(1+z_1z_2)z_3=z_1+z_3+z_1z_2z_3\neq 0
$$
which is the defining inequality of $X(\sigma^4;w_0)$.
\end{example}

\begin{example}
\label{ex: richardson}
Consider a pair of permutations $w,u\in S_n$ such that $w\geq u$ in the Bruhat order. Then 
$$
X(\beta(w)\beta(u^{-1}w_0);w_0)=R_{w,u}
$$
is the open Richardson variety in the complete flag variety in $\C^n$. Here $\beta(w)$ and $\beta(u^{-1}w_0)$ are positive braid lifts of permutations $w$ and $u^{-1}w_0$ respectively.
\end{example}

If in addition $w$ satisfies the so-called $k$-Grassmannian condition, then by \cite{KLS} the open Richardson variety $R_{w,u}$ is isomorphic to a more general positroid variety $\Pi_{w,u}\subset \mathrm{Gr}(k,n)$, the variety $\Pi_{k,n}$ corresponds to the case when $u = 1$ and $w$ is the maximal $k$-Grassmannian permutation. In \cite{CGGS2} the variety $\Pi_{w,u}$ is proven to be isomorphic (up to $\C^*$ factors) to braid varieties for four different braids, some on $n$ strands (like in Example \ref{ex: richardson}) and some on $k$ strands (like in Lemma \ref{ex: positroid}). We refer to \cite{CGGS2} for more details.

Finally, we would like to mention an emerging connection between braid varieties and {\bf cluster varieties}, although the latter are out of scope of these notes. In particular, by \cite{SSBW,GL3} all positroid varieties $\Pi_{w,u}$ are cluster varieties, while by \cite{GaoSW} certain braid varieties also have the structure of cluster varieties.

\begin{example}
\label{ex: cluster A}
The braid variety $X(\sigma^{k};w_0)$ corresponding to the $(2,k-1)$ torus link, is isomorphic (up to a certain torus) to the positroid variety $\Pi_{2,k+1}$, and to the cluster variety of type $A_{k-2}$.
\end{example}

\begin{example}
\label{ex: E6}
The braid variety $X((\sigma_1\sigma_2)^4\Delta;w_0)$ corresponding to the $(3,4)$ torus knot, is isomorphic (up to a certain torus) to the positroid variety $\Pi_{3,7}$, and to the cluster variety of type $E_6$.
\end{example}

It is yet unclear how the cluster structure reveals itself in link homology. However, one piece of this structure already plays an important role: any cluster variety carries a canonical closed 2-form \cite{GeShV} (which is symplectic under nice circumstances) which has constant coefficients in all cluster charts. This form yields an interesting operator in link homology of homological degree 2, which is constructed for all (not necessary positive) braids in the next section.

\section{$y$-ification and tautological classes}

In this lecture we describe various homological operations in Khovanov-Rozansky homology.

\subsection{Polynomial action}

Recall that to any braid $\beta$ we associate the Rouquier complex $T_{\beta}$ consisting of $R-R$ bimodules, or, equivalently, a complex of $\C[x_1,\ldots,x_n,x'_1,\ldots,x'_n]$-modules. Let us describe the properties of this complex.

First, observe that for any symmetric function $f(x_1,\ldots,x_n)\in \C[x_1,\ldots,x_n]^{S_n}$ the actions of $f(x_1,\ldots,x_n)$ and of $f(x'_1,\ldots,x'_n)$ on $T_{\beta}$ coincide. Indeed, this is true for $B_i$ and arbitrary products of $B_i$, and hence for any Soergel bimodule.

More abstractly, we can consider a bimodule
$$
B:=B_{w_0}=R\otimes_{R^{S_n}}R=\frac{\C[x_1,\ldots,x_n,x'_1,\ldots,x'_n]}{f(x_1,\ldots,x_n)=f(x'_1,\ldots,x'_n)\ \text{for any}\ f\in \C[x_1,\ldots,x_n]^{S_n}}
$$
Note that the bimodule $B$ is also an algebra as a quotient of the polynomial algebra. By the above, we obtain the following
\begin{lemma}
The action of $\C[x_1,\ldots,x_n,x'_1,\ldots,x'_n]$ on the Rouquier complex $T_{\beta}$ factors through the action of $B$.
\end{lemma}

Next, we can compare the actions of $x_i$ and $x'_i$. More precisely:
\begin{theorem}
\label{th: homotopy}
The action of $x_i$ on $T_\beta$ is homotopic to the action of $x'_{w(i)}$ where $w$ is the permutation corresponding to the braid $\beta$. 
\end{theorem}

\begin{example}
For example, consider the braid $\beta$:
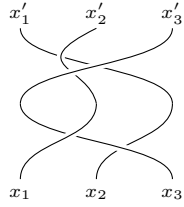
\begin{center}
\begin{tikzpicture}
\draw (2,0)..controls (2,0.5) and (3,0.5)..(3,1);
\draw [white, line width=3](3,0)..controls (3,0.5) and (1,0.5)..(1,1);
\draw (3,0)..controls (3,0.5) and (1,0.5)..(1,1);
\draw [white, line width=3] (1,0)..controls (1,0.5) and (2,0.5)..(2,1);
\draw (1,0)..controls (1,0.5) and (2,0.5)..(2,1);
\draw (3,1)..controls (3,1.5) and (1,1.5)..(1,2);
\draw [white, line width=3] (2,1)..controls (1.9,1.5) and (1,1.5) ..(2,2);
\draw (2,1)..controls (1.9,1.5) and (1,1.5) ..(2,2);
\draw [white, line width=3] (1,1)..controls (1,1.5) and (3,1.5)..(3,2);
\draw (1,1)..controls (1,1.5) and (3,1.5)..(3,2);
\draw (1,-0.2) node {\scriptsize $x_1$};
\draw (2,-0.2) node {\scriptsize $x_2$};
\draw (3,-0.2) node {\scriptsize $x_3$};
\draw (1,2.2) node {\scriptsize $x'_1$};
\draw (2,2.2) node {\scriptsize $x'_2$};
\draw (3,2.2) node {\scriptsize $x'_3$};
\end{tikzpicture}
\end{center}
We have $x_1\sim x'_2, x_2\sim x'_1$ and $x_3\sim x'_3$ on the corresponding Rouquier complex $T_{\beta}$. Note that after the braid closure we identify $x_i$ with $x'_i$, so that $x_1$ becomes homotopic to $x_2$, $x_2$ to $x_1$ and $x_3$ to itself. 
\end{example}

Theorem \ref{th: homotopy} has a very important consequence for the structure of link homology.

\begin{corollary}
Suppose that a braid $\beta$ closes up to a link $L$ with $r$ components. Then $\HHH(L)$ is naturally a module over a polynomial ring in $r$ variables, one variable per component of $L$.
\end{corollary}

Note that the components of $L$ correspond to the cycles in the corresponding permutation $w$.

\subsection{Dot-sliding homotopies and $y$-ification}

The following lemma outlines the proof of Theorem \ref{th: homotopy} for a single crossing. The general case is obtained by combining these elementary ``dot-sliding homotopies".

\begin{lemma} 
\label{lem: homotopy crossing}
The actions of  $x_i$ and $x'_{i+1}$  on the Rouquier complex $T_i=[B_i\to R]$ are homotopic.
\end{lemma}

\begin{proof}
We want to construct a map $h: R \to B_i$ satisfying $hd(f \otimes_{R^{s_{i}}}g) = x_{i}f\otimes_{R^{s_{i}}} g - f\otimes_{R^{s_{i}}} gx_{i+1}$ and $dh(f) = (x_i - x_{i+1})f$, where $d = m$ is the unique bimodule bap $B_{i} \to R$ sending $1 \otimes_{R^{s_{i}}} 1$ to $1$. Applying $hd(1 \otimes_{R^{s_{i}}} 1)$, it is clear that we must have $h(1) = x_{i} \otimes_{R^{s_{i}}}1 - 1 \otimes_{R^{s_{i}}} x_{i+1}$. This is exactly the map $\Delta$ constructed in Lemma \ref{lem: problem 2}. Now it is straightforward to check that $h$ indeed gives a homotopy between $x_i$ and $x'_{i+1}$. %We claim that this defines a bimodule map $h:R \to B_{i}$. By Problem 2 of Lecture 1, it suffices to show that $x_{i} \otimes_{R^{s_{i}}} 1 - 1 \otimes_{R^{s_{i}}} x_{i+1}$ coincides, up to a scalar multiple, with $\alpha_{s_{i}} \otimes_{R^{s_{i}}} 1 + 1 \otimes_{R^{s_{i}}} \alpha_{s_{i}}$, where $\alpha_{s_{i}} = x_{i+1} - x_{i}$. This is straightforward to check using the equality $(x_{i} + x_{i+1}) \otimes_{R^{s_{i}}} 1 = 1 \otimes_{R^{s_{i}}} (x_{i} + x_{i+1})$. It is equally straightforward to check now that $h$ indeed gives an homotopy between $x_{i}$ and $x_{i+1}'$.
\begin{center}
\begin{tikzcd}
B \arrow[r, "d"] & R \ar[dl, "x_i - x'_{i+1}" near start] \\
B \arrow[r, "d"] & R
\end{tikzcd}
\end{center}
\end{proof}

\begin{example}
\label{ex: sigma2}
Let us describe an explicit homotopy between the left and right $R$-actions for the two-strand braid $\sigma^2$.

Thanks to Lemma \ref{lem: homotopy crossing}, we have a homotopy between $x_1$ and $x'_2$ on $T$, so $(x_1 - x'_{2}) \otimes 1$ is null-homotopic on $T\otimes T$, such a homotopy is given by $h \otimes 1$. Similarly, $1 \otimes (x_{2} - x_{1}')$ is null-homotopic. But $x_{2}' \otimes 1 = 1 \otimes x_{2}$. Thus, $h \otimes 1 - 1 \otimes h$ gives a homotopy between $x_{1} \otimes 1$ and $1 \otimes x'_{1}$. It is easy to see that the same map $h \otimes 1 - 1 \otimes h$ gives a homotopy between $x_{2} \otimes 1$ and $1 \otimes x'_{2}$ (perhaps up to a scalar multiple). Alternatively, if we actually want to do computations, we have:
\begin{center}
    \begin{tikzcd}
    B \ar[r, "x_1 - x'_1"] & B \ar[r, "d"] \ar[dl, "1" near start] & R \ar[dl, "0" near start] \\ B \ar[r, "x_1 - x'_1"] & B \ar[r, "d"] & R
    \end{tikzcd}
\end{center}
The above complex works because, on $R$, $x_1 - x'_1 = x_2 - x'_2 = 0$ and, on $B$, $x_1 + x_2 = x'_1 + x'_2$, equivalently $x_1 - x'_1 = x'_2 - x_2$, so the diagram shows that both $x_1 - x'_1$ and $x'_2 - x_2$ are null-homotopic, as wanted.
\end{example}

Let $\xi_i$ be the homotopies between $x_i$ and $x'_{w_i}$, that is,
$$
[d,\xi_i]=x_i-x'_{w(i)}
$$
We can choose $\xi_i$ so that they square to zero and anticommute (since this holds for a single crossing by Lemma \ref{lem: homotopy crossing}). Now we can introduce formal variables $y_1,\ldots,y_n$ and consider the deformed differential 
\begin{equation}
    D=d+\sum_{i=1}^{n}\xi_i y_i.
\end{equation}
From the above discussion, we get 
$$
D^2=\sum_{i=1}^{n}(x_i-x'_{w(i)})y_i
$$
which vanishes after closing the braid and identifying different $y_i$ on the same link component. This allows one to define a deformation (or $y$-ification) of Khovanov-Rozansky homology:
$$
\HY(\beta)=H^*\left(\HH(T_{\beta}\otimes \C[y_1,\ldots,y_n],D)/(y_i-y_{w(i)})\right).
$$
It has the following properties:
\begin{theorem}[\cite{GH}]
The $y$-ified link homology is invariant under conjugation and stabilization and defines a topological link invariant. 
If a braid $\beta$ closes up to a link $L$ with $r$ components then $\HY(L)$ is naturally a module over $\C[x_1,\ldots,x_r,y_1,\ldots,y_r]$.
\end{theorem}

\begin{example}
\label{ex: yified homology two strands}
Let us compute the $y$-ified homology of $\sigma^2$. By Example \ref{ex: sigma2}, the $y$-fied complex for $\sigma^2$ has the form
\begin{center}
    \begin{tikzcd}
    B[y] \arrow[r,bend left, "d"] & B[y] \arrow[r, "d"] \arrow[l,bend left, "y_1-y_2"] & R[y]
    \end{tikzcd}
\end{center}
If we apply $\Hom(R,-)$ (or, equivalently, compute $y$-ified homology in $a$-degree zero) to this, we get 
\begin{center}
    \begin{tikzcd}
    R[y] \arrow[r,bend left, "0"] & R[y] \arrow[r, "x_1-x_2"] \arrow[l,bend left, "y_1-y_2"] & R[y]
    \end{tikzcd}
\end{center}
The homology of this complex is an $R[y]$-module with two generators (say, $z$ and $w$) and one relation $z(y_1-y_2)=w(x_1-x_2)$. This module is isomorphic to an ideal in $R[y]$ generated by $x_1-x_2$ and $y_1-y_2$.
\end{example}

\begin{remark}
The complex with backward pointing arrows might look unusual to the reader. We can rewrite it in a more conventional way by ``unrolling'' the variable $\mathbf{y}=y_1-y_2$, which has homological degree $2$, and writing $R[y]=\bigoplus_{k=0}^{\infty}\mathbf{y}^{k}R[y_1+y_2]$:
\begin{center}
\begin{tikzcd}
    R[y_1+y_2] \arrow{r}{0}& R[y_1+y_2]\arrow{r}{x_1-x_2} \arrow{dl}{\mathbf{y}}& R[y_1+y_2]\\
    \mathbf{y}R[y_1+y_2] \arrow{r}{0}& \mathbf{y}R[y_1+y_2] \arrow{r}{x_1-x_2} \arrow{dl}{\mathbf{y}} & \mathbf{y}R[y_1+y_2]\\
    \mathbf{y}^2R[y_1+y_2]\arrow{r}{0} & \mathbf{y}^2R[y_1+y_2] \arrow{r}{x_1-x_2} \arrow{dl}{}& \mathbf{y}^2R[y_1+y_2]\\
    \vdots & \vdots & \vdots\\
\end{tikzcd}
\end{center}
\end{remark}

This example can be generalized as follows:

\begin{theorem}[\cite{GH}]
\label{th: full twist}
Let $T(n,kn)$ be the $(n,kn)$ torus link with $n$ components. The corresponding braid is the $k$-th power of the full twist on $n$ strands ($k\ge 0$). Then:

(a) $\HY(T(n,kn))\simeq \CJ^k$ where 
$$
\CJ=\cap_{i\neq j}(x_i-x_j,y_i-y_j,\theta_i-\theta_j)
$$
is the ideal in $\HY(\mathrm{unlink})=\C[x_1,\ldots,x_n,y_1,\ldots,y_n,\theta_1,\ldots,\theta_n]$. In Hochschild degree zero we get
$$
\HY^{A=0}(T(n,kn))\simeq J^k,
$$
where
$$
J=\cap_{i\neq j}(x_i-x_j,y_i-y_j)\subset \C[x_1,\ldots,x_n,y_1,\ldots,y_n]
$$
is the ideal defining the union of all diagonals in $(\C^2)^n$.

(b) The ideals $\CJ$ and $J$ are free over $\C[y_1,\ldots,y_n]$ and $$
\HHH(T(n,kn))=\CJ^k/(y)\CJ^k,\ \HHH^0(T(n,kn))=J^k/(y)J^k.
$$
\end{theorem}

Note that in Theorem \ref{th: full twist} it is much easier to first describe the $y$-ified homology, and then obtain $\HHH$ as a quotient by the maximal ideal in $y_i$. This seems to indicate that $y$-ified homology has better properties than $\HHH$, as we see in the next section.

\subsection{Tautological classes and symmetry}

Consider the dg algebra $\CA$ which is a free resolution of $R$ over $B=R\bigotimes_{R^{S_n}}R$:
$$
\mathcal{A}=B[u_1,\ldots,u_n,\xi_1,\ldots,\xi_n],\ d(\xi_i)=x_i-x'_i,\ d(u_k)=\sum_{i=1}^{n}h_{k-1}(x_i,x'_i)\xi_i,
$$

\noindent where $h_{k-1}(x_i, x'_i)$ denotes the complete symmetric function of degree $k-1$ on the variables $x_i, x'_i$, and the grading on $\mathcal{A}$ is given by $\deg(x_i) = \deg(x'_i) = 0$, $\deg(\xi_i) = 1$, and $\deg(u_i) = 2$ for $i = 1, \dots, n$.

\begin{lemma} 
We have $d^2=0$ in $\mathcal{A}$. 
\end{lemma}
 
\begin{proof}
First note that thanks to the graded Leibniz rule, it is enough to show that $d^{2} = 0$ on generators. Indeed, if $b$ and $c$ are homogeneous then
\[
\begin{array}{rl}
d^{2}(bc) =&  d(d(b)c + (-1)^{|b|}bd(c)) = d^{2}(b)c + (-1)^{|d(b)|}d(b)d(c) + (-1)^{|b|}d(b)d(c) + bd^{2}(c)  \\ & =  d^2(b)c + bd^2(c)
\end{array}
\]
 where the last equality follows since $|b| = |d(b)| - 1$. Now, that $d^{2}(x_{i}) = d^{2}(x_{i}') = d^{2}(\xi_{i}) = 0$ is obvious from the definition. As for $u_{k}$ we have
\[
d^2(u_{k}) = \sum_{i = 1}^{n}h_{k-1}(x_{i}, x_{i}')(x_{i} - x_{i}') = \sum_{i = 1}^{n} x_{i}^{k} - (x_{i}')^{k} = 0.
\]
\end{proof}

\begin{theorem}[\cite{GHM}]
\label{th: A acts}
The algebra $\CA$ acts on the Rouquier complex $T_{\beta}$ for any $\beta$, such that $x_i$ act as usual, the action of $x'_i$ is twisted by the permutation $w$ corresponding to $\beta$, and $\xi_i$ act by the dot-sliding homotopies as above.
\end{theorem}

Theorem \ref{th: A acts} follows from two lemmas below. Indeed, by Lemma \ref{lem: A acts on a single crossing} the dg algebra $\CA$ acts on Rouquier complexes $T_i^{\pm}$ (such that $u_i$ act by 0), and by Lemma \ref{lem: A coproduct} there is a coproduct on $\CA$ which allows to define the structure of $\CA$-module on the tensor product $M\otimes_{R} N$ of two $\CA$-modules. Note that in general $u_i$ act nontrivially on $T_{\beta}$, see Example \ref{ex: A sigma2}.

\begin{lemma}
\label{lem: A coproduct}
There is a chain map $\mathcal{A}\to \mathcal{A}\otimes_{R}\mathcal{A}$ defined by the equations
$$
\Delta(x_i)=x_i\otimes 1,\ \Delta(x'_i)=1\otimes x'_i,\ \Delta(\xi_i)=\xi_i\otimes 1+1\otimes \xi_i
$$
$$
\Delta(u_k)=u_k\otimes 1+1\otimes u_k+\sum_{i=1}^{n}h_{k-2}(x_i,x'_i,x''_i)\xi_i\otimes \xi_i
$$
\end{lemma}

\begin{proof}
We need to check that $\Delta$ commutes with the differentials. This is clear for $x_{i}$. Let us check for $\xi_{i}$:
$$
d(\Delta(\xi_{i})) = d(\xi_i\otimes 1 + 1 \otimes \xi_i) = (x_{i} - x_{i}') \otimes 1 + 1 \otimes (x_i - x_{i}')
$$
\noindent but, just as before, $x_{i}' \otimes 1 = 1 \otimes x_{i}$, and $x_{i}''$ is just a shorthand for $1 \otimes x_{i}'$. So we see that
$$
d(\Delta(\xi_{i})) = x_{i} - x_{i}'' = \Delta(d(\xi_{i})).
$$

Finally, let us compute $d(\Delta(u_{k}))$:
\[\begin{array}{rl}
d(\Delta(u_{k})) & = d(u_k \otimes 1 + 1 \otimes u_k + \sum_{i = 1}^{n}h_{k-2}(x_i, x_i', x_i'')\xi_{i} \otimes \xi_{i}) \\
& = \sum_{i=1}^{n}[h_{k-1}(x_i, x_i')\xi_i \otimes 1 +
 1 \otimes h_{k-1}(x_i, x_i')\xi_i +\\ & h_{k-2}(x_i, x'_i, x''_i)((x_i - x_i')\otimes \xi_{i} - \xi_{i}\otimes(x_i - x'_i))]
\end{array}
\]

\noindent where the last negative sign is due to the graded Leibniz rule. Recalling that $x_i' = x_i' \otimes 1 = 1 \otimes x_i$ and $x''_i = 1 \otimes x'_i$ we can rewrite as 
\[
\begin{array}{l}
\sum_{i = 1}^{n}[h_{k-1}(x_i, x_i')\xi_i \otimes 1 + h_{k-1}(x_i', x_i'')1 \otimes \xi_i +\\ h_{k-2}(x_i, x_i', x_i'')((x_i - x_i')(1 \otimes \xi_i) - (x_i' - x_i'')(\xi_i \otimes 1))] = \\
\sum_{i = 1}^{n}[(h_{k-1}(x_i, x'_i)+h_{k-2}(x_i, x_i', x_i'')(x_i''-x'_i))(\xi_i\otimes 1) +\\ (h_{k-1}(x'_i, x''_i)+h_{k-2}(x_i, x_i', x_i'')(x_i - x_i'))(1 \otimes \xi_i)]
\end{array}
\]

Now we observe that $h_{k-1}(x_i, x_i') + h_{k-2}(x_i, x'_i, x''_i)(x''_i - x'_i) = h_{k-1}(x_i, x''_i)$ and, similarly, $h_{k-1}(x'_i, x''_i)+ h_{k-2}(x_1, x'_i, x''_i)(x_i - x'_i) = h_{k-1}(x_i, x''_i)$.  Then,
$$
d(\Delta(u_k)) = \sum_{i=1}^{n}h_{k-1}(x_i, x''_i)(\xi_i \otimes 1 + 1 \otimes \xi_i) = \Delta(d(u_k)).
$$
\end{proof}

\begin{lemma}
\label{lem: A acts on a single crossing}
The dg algebra $\mathcal{A}$ acts on the Rouquier complex $T_i$.
\end{lemma}

\begin{proof}
By degree considerations, all $u_{k}$ have to act by $0$ on $T_{i}$. The action of $x_{i}$ is given by the left action of $R$ on $T_{i}$, while the action of $x'_{i}$ is given by the right action of $R$ \emph{twisted by the automorphism $s_{i}: R \to R$}, so that $x'_i$ is given by right multiplication by $x_{i+1}$, for example. Finally, the action of $\xi_{i}$ is given by the homotopy $h$ constructed in Lemma \ref{lem: homotopy crossing}, the action of $\xi_{i+1}$ is the negative of this homotopy, and $\xi_{k}$ acts by $0$ for $k \neq i, i+1$. Let us verify that this indeed gives an action of $\mathcal{A}$ on $T_{i}$.

Since all the maps appearing in the Rouquier complex are maps of $R$-bimodules, they commute with both the left and right action of $R$, as needed. Any symmetric polynomial is, in particular, symmetric in the $i$ and $i+1$-st variables, so the left and right actions of symmetric polynomials on $T_{i}$ coincide. This verifies that we have an action of $B = R \otimes_{R^{S_{n}}}R$ with trivial differential. 

By definition of the action and Lemma \ref{lem: homotopy crossing}, we have $d\xi_{i} = x_{i} - x'_{i}$, and $d\xi_{i+1} = x_{i+1} - x'_{i+1}$. The twist $s_{i}$ fixes every other variable, and the left and right actions of every other variable on $T_{i}$ and on $R$ coincide. Thus, for $k \neq i, i+1$ we get $d\xi_{k} = 0 = x_{k} - x'_{k}$, as needed. 

Finally, we need to check that $\sum_{\ell = 1}^{n}h_{k-1}(x_\ell, x'_\ell)\xi_{\ell}$ acts by $0$. Since $\xi_{\ell}$ acts by $0$ unless $\ell = i, i+1$, we simply get $h_{k-1}(x_{i}, x_{i}')\xi_{i} + h_{k-1}(x_{i+1}, x'_{i+1})\xi_{i+1}$. But the action is defined so that $\xi_{i+1} = -\xi_{i}$, and this further simplifies as $(h_{k-1}(x_i, x_{i'}) - h_{k-1}(x_{i+1}, x'_{i+1}))\xi_{i+1}$. On $B_{i}$ this is zero because $\xi_{i+1}$ is already zero there. On $R$, we get:
$$
h_{k-1}(x_i, x'_{i}) - h_{k-1}(x_{i+1}, x'_{i+1}) = h_{k-1}(x_i, x_{i+1}) - h_{k-1}(x_{i+1}, x_{i}) = 0.
$$
\end{proof}

\begin{example}\label{ex: A sigma2}
Let us construct the action of the dg algebra $\mathcal{A}$ on the Rouquier complex for $\sigma^2$. 

Thanks to Lemma \ref{lem: A acts on a single crossing}, we get an action of $\cA \otimes \cA$ on $T \otimes T = \sigma^2$. Composing with the coproduct map of Lemma \ref{lem: A coproduct} , this induces an action of $\cA$ on $T \otimes T = \sigma^2$. Let us be more explicit. The action of $u_1$ has to be $0$, because $u_1$ acts by zero on $T$ and $\Delta(u_1) = u_1 \otimes 1 + 1 \otimes u_1$. The action of $x_i$ (resp. $x'_i$) is given by left (resp. right) multiplication by $x_i$. The action of $\xi_1$ is the homotopy of Example \ref{ex: sigma2} and the action of $\xi_2$ the negative of this. Finally, by degree considerations $u_2$ can only be nonzero on $R$, and there it is the map $R \to B$ constructed in Lemma \ref{lem: problem 2}.

\begin{center}
    \begin{tikzcd}
    B \ar[rr, "x_1 - x'_1"] & & B \ar[rr] \ar[ll, bend left, "\xi_1"] & & R \ar[llll, "u_2", bend right] 
    \end{tikzcd}
\end{center}

That $dx_1 = dx_2 = dx'_1 = dx'_2 = 0$ follows because all the maps involved in the complex $B \to B \to R$ are bimodule homomorphisms. That $d\xi_{1} = x_1 - x'_1$ and $d\xi_2 = -d\xi_1 = x_2 - x'_2$ is essentially Example \ref{ex: sigma2} above. So we only need to check the relation $du_2 = (x_1 + x'_1)\xi_1 + (x_2 + x'_2)\xi_2$. Note that the right-hand side is zero, while the left-hand side is the composition of the maps $R \longrightarrow B \buildrel {x_1 - x'_1} \over \longrightarrow B$. To check that it is zero, it suffices to check that it is zero when evaluated at $1$. Here, we have:
$$
\begin{array}{rl}
(x_1 - x'_1)(x_1 \otimes_{R^{s}} 1 - 1 \otimes_{R^{s}} x_2) = & x_1^2\otimes_{R^{s}} 1 - x_1 \otimes_{R^{s}} x_2 - x_1 \otimes_{R^{s}} x_1 + 1 \otimes_{R^{s}} x_1x_2 \\ = & (x_1^2 + x_1x_2)\otimes_{R^{s}}1 -x_1\otimes_{R^{s}}(x_1 + x_2) \\ = & (x_1^2 + x_1x_2 -x_1(x_1+x_2))\otimes_{R^{s}} 1 = 0.
\end{array}
$$
\end{example}

\begin{remark}
 Note that, similarly to Example \ref{ex: A sigma2}, $u_1$ acts by zero on any Rouquier complex $T_{\beta}$.
\end{remark}
To sum up Theorem \ref{th: A acts}, there is an action of interesting operators $\xi_i,u_k$ on Rouquier complexes $T_{\beta}$ but these do not commute with the differential on $T_{\beta}$. To resolve this issue, we use the $y$-ified complex with the differential
$$
D=d+\sum_{i=1}^{n}\xi_{i}y_i
$$
and define 
$$
F_k=u_k+\sum_{i=1}^{n}h_{k-1}(x_i,x'_i)\frac{\partial}{\partial y_i}.
$$
Then we get the following:

\begin{theorem}[\cite{GHM}]
\label{th: tautological}
(a) We have
$$
[D,F_k]=0,\ [F_k,F_l]=0,
$$
$$
[F_k,x_i]=0,\ [F_k,y_i]=h_{k-1}(x_i,x'_i).
$$
In particular, $F_k$ define a family of commuting operators on $\HY(\beta)$.

(b) The operator $F_2$ satisfies a ``hard Lefschetz" condition and lifts to an action of $\mathfrak{sl}_2$ on $\HY(\beta)$. As a corollary, $\HY(\beta)$ is symmetric.
\end{theorem}

For knots, we have $\HY(\beta)=\HHH(\beta)\otimes \C[x,y]$, so $\HHH(\beta)$ is symmetric as well. This was conjectured by Dunfield, Gukov and Rasmussen in \cite{DGR}, see Example \ref{ex: T34 homology} for the visually symmetric representation of $\HHH(T(3,4))$, and Examples \ref{ex: two strand homology} and \ref{ex: t22} for $\HHH(T(2,n))$.

Note that the ``curious hard Lefschetz" property in the homology of positroid varieties was established in \cite{GL3,LS}.

\subsection{Geometric analogue}

The construction of the operators $u_k$ and $F_k$ via the coproduct on $\CA$ is similar to (and motivated by) the construction of the tautological classes on character varieties \cite{Boalch,BSS,Jeffrey,Mellit} which we briefly recall in this section. 
Let $G=GL(n)$ and let $Q$ be a symmetric function in $n$ variables of degree $r$. Then one can use the Bott-Shulman-Stasheff construction \cite{Jeffrey} to obtain the following differential forms and cohomology classes:
$$
\Phi_0(Q)\in  H^{2r}(BG),\ \Phi_1(Q)\in \Omega^{2r-1}(G),\  \Phi_2(Q)\in \Omega^{2r-2}(G\times G)\ldots
$$
such that $d\Phi_1(Q)=0$ (so that $\Phi_1(Q)$ defines a cohomology class on $G$) and
$$
d\Phi_2(Q)=\pi_1^*\Phi_1(Q)+\pi_2^*\Phi_1(Q)-m^*\Phi_1(Q)
$$
where $m:G\times G\to G$ is the multiplication map. 

One can think of these classes as follows: first, recall that the cohomology of $BG$ is isomorphic to the $G$-equivariant cohomology of a point, and to $\C[\mathfrak{g}]^{G}\simeq \C[x_1,\ldots,x_n]^{S_n}$, so a symmetric function $Q$ naturally defines a cohomology class $\Phi_0(Q)$. In particular, the cohomology of $BG=BGL(n)$ is a free polynomial algebra generated in degrees $2,4,\ldots,2n$.

Next, we consider the Leray spectral sequence 
$$
H^*(G)\otimes H^*(BG)\Rightarrow H^*(EG)=\C
$$
associated to the universal fibration $EG\to G$. The class $[\Phi_1(Q)]\in H^{2r-1}(G)$ is characterized by the fact that it kills $\Phi_0(Q)\in H^{2r}(BG)$ in this spectral sequence. We refer to \cite{Jeffrey} for an explicit construction of the differential form $\Phi_1(Q)$ representing this cohomology class.
In particular, the cohomology of $GL(n)$ is a free polynomial algebra generated by anticommuting variables in degrees $1,3,\ldots,2n-1$.

The form $\Phi_2(Q)$ corresponds to the equation
\begin{equation}
\label{eq: coproduct G}
m^*[\Phi_1(Q)]=1\otimes [\Phi_1(Q)]+[\Phi_1(Q)]\otimes 1.
\end{equation}
Although this equation holds in cohomology (where $m^*:H^*(G)\to H^*(G)\otimes H^*(G)$), it does not hold for the actual differential forms $\Phi_1(Q)$ on the nose. The difference between the left and right hand sides in \eqref{eq: coproduct G} can be then written as 
$d\Phi_2(Q)$ and it turns out that such $\Phi_2(Q)$ can be written explicitly. We refer to \cite{BSS,Jeffrey,Mellit} for more details and explicit formulas for $\Phi_1(Q)$ and $\Phi_2(Q)$.

\begin{example}\label{ex: square sum}
If $Q=\sum x_i^2$ then 
$$
\Phi_2(Q)=\Tr(f^{-1}df\wedge dg g^{-1}).
$$
\end{example}

These forms can be used for ``gluing'' various $G$-valued functions. Given $f:X\to G$ and $g:Y\to G$ for some $X$ and $Y$, assume that 
$$
f^*\Phi_1(Q)=d\omega_X,\ g^*\Phi_1(Q)=d\omega_Y,
$$
then we can define a new form
\begin{equation}
\label{eq: glue forms}
\omega_{X\times Y}=\omega_X+\omega_Y+(f\times g)^*\Phi_2(Q)\in \Omega^{2r-2}(X\times Y).
\end{equation}
It follows from the above that 
$$
d\omega_{X\times Y}=(f\cdot g)^*\Phi_1(Q).
$$
Equation \eqref{eq: glue forms} is similar to the construction of the coproduct on $\CA$. Mellit \cite{Mellit2} used this construction with $Q = \sum x_i^2$, cf. Example \ref{ex: square sum}, to define a  2-form on an arbitrary braid variety $X(\beta)$. Note that the $2$-form on $X(\beta)$ is closed, since $\Phi_2(Q)$ vanishes on upper-triangular matrices, but we cannot expect it to be symplectic in general, since $\dim(X(\beta))$ is not necessarily even. However, Mellit takes a closed subvariety $\widetilde{Y}(\beta) \subseteq X(\beta)$ by fixing the diagonal of the matrix $B_{\beta}$, and takes the quotient $Y(\beta) := \widetilde{Y}(\beta)/\widetilde{T}$ by the action of an appropriate torus $\widetilde{T}$, see \cite[Definition 5.3.7]{Mellit2} for details.

\begin{theorem}[\cite{Mellit2}]
\label{th: sl2 braid variety}
The above construction produces a symplectic form on $Y(\beta)$ which satisfies ``curious hard Lefschetz" with respect to the weight filtration in cohomology. 
\end{theorem}

\section{Algebraic links and affine Springer theory}

\subsection{Algebraic links}

Let $f(x,y)$ be a polynomial in two variables. Consider the plane curve  $C=\{f(x,y)=0\}\subset \C^2$. We assume that $C$ passes through the origin and has a singular point there. 

Consider the intersection $L=C\cap S^3_{\varepsilon}$ of $C$ with a small three-sphere with center at the origin and radius $\varepsilon$. It is a classical result of Milnor \cite{Milnor} that for $\varepsilon$ small enough the intersection is transverse (so that $L$ is a smooth link in $S^3$) and the topological type of $L$ does not depend on $\varepsilon$. Such links are called {\bf algebraic links}.

\begin{example}
The node $\{xy=0\}$ corresponds to the Hopf link (that is, the $(2,2)$ torus link). The cusp $\{x^2=y^3\}$ corresponds to the trefoil knot, that is, $T(2,3)$. More generally, the singularity $\{x^m=y^n\}$ corresponds to a positive $(m,n)$ torus link which has $d=GCD(m,n)$ components. 
\end{example}

It is known that (local) irreducible components of $C$ at the origin correspond to the connected components of the link $L$. In particular, if $C$ is irreducible (and reduced) then $L$ is an algebraic knot with one component. Such knots are classified in \cite{EN} and are all iterated cables of torus knots. The cabling parameters correspond to the Puiseux expansion of $C$. For links with more components, the classification is more complicated, and we refer to \cite{EN} for all details. 

\subsection{Oblomkov-Rasmussen-Shende conjectures}

Oblomkov, Rasmussen and Shende proposed a remarkable conjecture relating the singular curves to Khovanov-Rozansky homology. Recall that the Hilbert scheme of $n$ points on $C$ consists of ideals $I\subset \CO_C$ such that $\dim \CO_{C}/I=n$. We will also consider the local version of the Hilbert scheme where the ideals are contained in $\CO_{C,0}=\frac{\C[[x,y]]}{f(x,y)}$.

\begin{conjecture}[\cite{ORS}]
\label{conj: ORS}
One has
$$
\HHH^0(L)=\bigoplus_{k=0}^{\infty}H^*(\Hilb^k(C,0)).
$$
Here the right hand side is bigraded by the number of points $k$ and the homological degree, and the two gradings are related to the gradings on the left hand side by an explicit change of variables. 
\end{conjecture}

\begin{remark}
In \cite{ORS} there is also a conjectural model for higher Hochschild degrees using slightly more complicated moduli spaces.
\end{remark}

In the next few examples, we will verify Conjecture \ref{conj: ORS} in some special cases. 

\begin{example}\label{ex: hopf}
Let us describe the Hilbert schemes $\Hilb^n(C,0)$ for the node $\{xy=0\}$, which corresponds to the Hopf link $L = T(2,2)$. 

First, we note that a (topological) basis of $\mathcal{O}_{C, 0}$ is given by the monomials $1, x, x^2, x^3, \dots, y, y^2, y^3, \dots$. Now we claim that a nonzero ideal $I \subseteq \mathcal{O}_{C, 0}$ satisfies exactly one of the three following properties:

\begin{itemize}
    \item[(a)] It is properly contained in $\C[[x]] \subseteq \mathcal{O}_{C, 0}$.
    \item[(b)] It is properly contained in $\C[[y]] \subseteq \mathcal{O}_{C, 0}$.
    \item[(c)] It has finite codimension. 
\end{itemize}

The ideals in (a) are precisely those of the form $(x^{n})$ for $n > 0$, they are properly contained in $\C[[x]]$ thanks to the relation $yx = 0$ and they do not have finite codimension, as $\C[[y]] \subseteq \mathcal{O}_{C, 0}/(x^n)$. The ideals in (b) are similarly described. Finally, if an ideal $I$ is not properly contained in $\C[[x]]$ nor in $\C[[y]]$, then there exist $m, n$ such that $x^{m}, y^{n} \in I$. But then the quotient $\mathcal{O}_{C, 0}/I$ is spanned by $1, x, \dots, x^{m-1}, y, \dots, y^{n-1}$, so $I$ has finite codimension. We will now focus on ideals satisfying (c). 

We claim that every ideal of finite codimension is either principal, of the form $(\alpha x^{i} + \beta y^{j})$ for $\alpha \neq 0, \beta \neq 0$ or an ideal of the form $(x^{i}, y^{j})$. Assume first that $I$ is principal, so $I = (f)$. We may also assume that $I \neq \mathcal{O}_{C, 0}$. Since $I$ has finite codimension, $f = g(x) + h(y)$, where $g(x), h(y) \neq 0$. Multiplying by a unit in $\C[[x]]$ and then by a unit in $\C[[y]]$, we see that $f$ may be assumed to be of the form $\alpha x^{i} + \beta y^{j}$, where $i = \nu(g), j = \nu(h)$ are the valuations of $g$ and $h$, that is, the minimal power of $x$ (resp. $y$) that appears with nonzero coefficient in $g$ (resp. $h$). 

Now assume that $I$ is not principal, say $I = (\alpha_{k}x^{i_k} + \beta_{k}y^{j_{k}})$ for some $k > 1$. Note that $(\alpha x^{i} + \beta y^{j}) \supseteq (\alpha'x^{m} + \beta' y^{n})$ if $i < m$ and $j < n$. From here, we can see that every non-principal ideal is of the form $(x^{i}, y^{j})$. %If we take $i := \min(i_{k})$ and $j := \min(j_{k})$ we can see that $I$ is, in fact, $(x^{i}, y^{j})$. 

Now, $\mathcal{O}_{C, 0}/(x^{i}, y^{j})$ has basis $1, x, \dots, x^{i-1}, y, \dots, y^{j-1}$, so $(x^{i}, y^{j}) \in \operatorname{Hilb}^{i+j-1}(C, 0)$. On the other hand, note that $(\alpha x^{i} + \beta y^{j})$ contains both $x^{i+1}$ and $y^{j+1}$. From here, we can see that $(\alpha x^{i} + \beta y^{j}) \in \operatorname{Hilb}^{i+j}(C, 0)$. 

Finally, since all the ideals of the form $(\alpha x^{i} + \beta y^{j})$ with $\alpha, \beta \neq 0$ contain $y^{j+1}$ and $x^{i+1}$, we see that
$$
\lim_{\alpha \to 0}(\alpha x^{i} + \beta y^{j}) = (x^{i+1}, y^{j}), \qquad  \lim_{\beta \to 0}(\alpha x^{i} + \beta y^{j}) = (x^{i}, y^{j+1})
$$
\noindent so that $\{(\alpha x^{i} + \beta y^{j}), (x^{i+1}, y^{j}), (x^{i}, y^{j+1}) \mid \alpha, \beta \neq 0\} = \mathbb{P}^{1} \subseteq \Hilb^{i+j}(C, 0)$, and we get that for $n > 1$, $\Hilb^{n}(C, 0)$ is a chain of $n-1$ projective lines $\mathbb{P}^{1}$, dual to the $A_{n}$ Dynkin diagram, where each $\mathbb{P}^{1}$ corresponds to a way of writing $n$ as a sum of two positive integers. The remaining cases, $\Hilb^{1}(C, 0)$ and $\Hilb^{0}(C, 0)$ are clearly just points. 

Let us verify Conjecture \ref{conj: ORS} in this case. The generating function for the Poincar\'e polynomials of $\Hilb^k(C,0)$ equals
$$
\sum_{k,i}q^kt^i\dim H^i(\Hilb^k(C,0))=1+q+q^2(1+t^2)+q^3(1+2t^2)+\ldots=\frac{1}{1-q}+\frac{q^2t^2}{(1-q)^2}.
$$
This agrees with the Poincar\'e polynomial of $\HHH^0(L)$, computed in Example \ref{ex: t22}, % $\HHH^0(L)=R(-4)[-2]\oplus R/(x_1-x_2)$ where $R=\C[x_1,x_2]$
up to the change of variables $t\mapsto 1/qt^2$  and the computation in Example \ref{ex: two strand recursion} up to the change of variables $t\mapsto 1/q^2t$. %\fixme{check and fix this after gradings in above sections are corrected}
\end{example}

\begin{example}\label{ex: cusp}
Let us describe the Hilbert schemes for the cusp $\{x^3=y^2\}$ which we can parametrize by $x=t^2,y=t^3$ so that $\CO_{C,0}=\C[[t^2, t^3]]$.

First, we claim that a nonzero ideal $I \subseteq \C[[t^2, t^3]]$ has finite codimension. Let $I$ be such an ideal and $0 \neq f \in I$. Let $g(t) \in \C[[t]]$ be a unit such that $gf = t^{k}$. Note that $g(t)$ may not belong to $\C[[t^2, t^3]]$ but $t^2g$ does. Thus, $t^{k+2} \in I$, and we have proved that every nonzero ideal contains a monomial. It remains to observe that $\C[[t^2, t^3]]/(t^{k})$ has $1, t^2, t^3, \dots, t^{k-1}, t^{k+1}$ as a basis, so every ideal generated by a monomial is finite codimensional and the claim follows. 

Now we claim that every proper ideal in $\C[[t^2, t^3]]$ is either principal, of the form $(t^k + \lambda t^{k+1})$ for some $\lambda \in \C$, or of the form $(t^k, t^{k+1})$. Assume first that $I$ is principal, say $I = (f)$. Let $k = \nu(f)$. We have seen in the paragraph above that $t^{k+2}, t^{k+3}, \dots \in I$, so we may assume $f$ has the form $t^{k} + \lambda t^{k+1}$, as needed. Now note that $(t^{\ell} + \mu t^{\ell + 1}) \subseteq (t^{k} + \lambda t^{k+1})$ if $\ell > k + 1$. From here, it follows that the non-principal ideals have the form $(t^k, t^{k+1})$.

Now note that $\C[[t^2, t^3]]/(t^{k}, t^{k+1})$ has basis $1, t^2, t^3, \dots, t^{k-1}$, so we have that $(t^{k}, t^{k+1}) \in \Hilb^{k-1}(C, 0)$. On the other hand, $\C[[t^2, t^3]]/(t^{k} + \lambda t^{k+1})$ has basis $1, t^2, t^3, \dots, t^{k-1}, t^{k+1}$, and $(t^{k} + \lambda t^{k+1}) \in \Hilb^{k}(C, 0)$.

To conclude, we have that $\Hilb^{0}(C, 0)$ and $\Hilb^{1}(C, 0)$ are both points, while for $k \geq 2$
$$
\Hilb^{k}(C, 0) = \{(t^{k} + \lambda t^{k+1}), (t^{k+1}, t^{k+2}) \mid \lambda \in \C\} = \mathbb{P}^{1}.
$$
The generating function has the form
$$
1+q+q^2(1+t^2)+q^3(1+t^2)+\ldots=\frac{1+q^2t^2}{1-q}.
$$

\noindent which again coincides with the Poincar\'e polynomial of $\HHH^{0}(T(2,3))$, computed in Example \ref{ex: t22} up to the change of variables $t \mapsto 1/q^2t$. 
\end{example}

\begin{example}\label{ex: hilb y2=x2k+1}
Similarly to Example \ref{ex: cusp}, let us describe the Hilbert schemes for the curve $\{x^{2k+1}=y^2\}$. Let us, first, classify the principal ideals in $\C[[t^2, t^{2k+1}]]$. Just as in Example \ref{ex: cusp}, we can see that if $I = (f)$ is a principal ideal then we may assume $f$ has the form $f = t^{m} + a_{1}t^{m+1} + \cdots + a_{2k-1}t^{m + 2k - 1}$. Note, however, that we may multiply $f$ by powers of $t^{2}$ to get rid of the monomials of the form $t^{m + 2i}$. Thus, the principal ideals of $\C[[t^2, t^{2k+1}]]$ have the form:
 $$
 \begin{array}{l}
 (t^2 + \lambda_1t^{2k+1}) \in \Hilb^2(C, 0) \\
 (t^{4} + \lambda_1 t^{2k+1} + \lambda_2 t^{2k+3}) \in \Hilb^4(C, 0) \\
 \vdots \\
 (t^{2k - 2} + \lambda_1 t^{2k+1} + \cdots + \lambda_{k-1}t^{2k -2 + 2k - 1}) \in \Hilb^{2(k-1)}(C,0) \\
 (t^{m} + \lambda_1t^{m+1} + \lambda_t^{m+3} + \cdots + \lambda_{k}t^{m + 2k - 1}) \in \Hilb^{m}(C, 0)
 \end{array}
 $$
 
 \noindent where $m \geq 2k$.  \\
 
 Assume now that $I$ is not principal, say $I = (f_{i})_{i = 1}^{\ell}$. We may assume that each $f_{i}$ has the form stated above, that is, $f_{i} = t^{m_{i}} + \lambda_{i,1}t^{m_{i}+1} + \cdots + \lambda_{i, k}t^{m_{i} + 2k - 1}$. Note that if $m_{1} := \min(m_{i})$, then $I$ already contains all monomials of the form $t^{m_{1} + 2m + j}$. So we may assume that $m_{1} \leq m_{i} \leq m_{1} + 2k - 1$ for all $i$. Moreover, if $m_{i} = m_{j}$ then we can replace $(f_{i}, f_{j})$ by $(f_{i}, f_{j} - f_{i})$ and we get rid of the $t^{m_{j}}$ term in $f_{j}$. In other words, we may assume that $m_{i} \neq m_{j}$ if $i \neq j$. Moreover, if $m_{j} = m_{i} + 2p$, then we can substitute $(f_{i}, f_{j})$ with $(f_{i}, f_{j} - t^{2p}f_{i})$ to get rid of the $t^{m_{j}}$ term in $f_{j}$. So we may assume that all $m_{i}$'s have different parity. To recap: an ideal is generated by at most $2$ elements. We have
$$
\begin{array}{l}
(t^{2}, t^{2k+1}) \in \Hilb^{1}(C, 0) \\
(t^{4}, t^{2k+1}) \in \Hilb^{2}(C, 0), (t^{4} + \lambda_{1}t^{2k+1}, t^{2k+3}) \in \Hilb^{3}(C, 0) \\
(t^{6}, t^{2k+1}) \in \Hilb^{3}(C, 0), (t^{6} + \lambda_{1}t^{2k+1}, t^{2k+3}) \in \Hilb^{4}(C, 0),\\
(t^{6} + \lambda_{1}t^{2k+1} + \lambda_{2}t^{2k+3}, t^{2k+5}) \in \Hilb^{5}(C,0) \\
\vdots \\
(t^{m}, t^{m+1}) \in \Hilb^{m-k}(C, 0) \\
(t^{m} + \lambda_1 t^{m+1}, t^{m+3}) \in \Hilb^{m-k+1}(C, 0) \\
\vdots \\
(t^{m} + \lambda_1 t^{m+1} + \cdots + \lambda_{k-1}t^{m+2k-3}, t^{m+2k-1}) \in \Hilb^{m-1}(C, 0)
\end{array}
$$

Let us briefly justify why we can assume the second generator, of higher degree, is simply a monomial. We do this in the case $I = (t^{m}, t^{m+1})$, which is the most involved one. We will make heavy use of the fact that $t^{m+2k}\C[[t]] \subseteq I$. A priori, we have the ideal
$$
(t^{m} + \lambda_1 t^{m+1} + \cdots + \lambda_{k}t^{m-2k+1}, t^{m+1} + \mu_1 t^{m+2} + \mu_2 t^{m+4} + \cdots + \mu_{k-1}t^{m + 2k - 2})
$$
We can multiply $f:= t^{m} + \lambda_1 t^{m+1} + \cdots + \lambda_{k}t^{m-2k+1}$ by $\mu_{k-1}t^{2k-2}$ and subtract to $g := t^{m+1} + \mu_1 t^{m+2} + \mu_2 t^{m+4} + \cdots + \mu_{k-1}t^{m + 2k - 2}$ to get rid of the $\mu_{k-1}t^{m+2k-2}$ term. This introduces a monomial of the form $t^{m+2k-1}$ in $g$. But now we can multiply $g$ by a scalar multiple of $t^{2k-2}$ to get rid of this term, too. We have substituted $g$ by $g_1 = t^{m+1} + \mu_1 t^{m+2} + \mu_2 t^{m+4} + \cdots + \mu_{k-2}t^{m + 2k - 4}$. Note that the coefficients $\mu_1, \mu_2, \dots, \mu_{k-2}$ have not changed.

Now we can do the same trick and multiply $f$ by an appropriate scalar multiple of $t^{2k-4}$ to get rid of the $t^{2k-4}$ term in $g_{1}$. This will introduce $t^{m+2k-3}$ and $t^{m+2k-1}$ terms. As we have seen before, we can easily get rid of the $t^{m+2k-1}$ term. If we want to get rid of the $t^{m+2k-3}$ term, we will introduce a $t^{m+2k-2}$ term, without introducing any new $t^{m+2k-4}$ terms. But we have seen that we can get rid of the $t^{m+2k-2}$ term. We can then use recursion to reduce $g$ to simply $t^{m+1}$. Then use this monomial to simplify $f$ to just $t^{m}$. Thus, we get: $\Hilb^{0}(C, 0)$ and $\Hilb^{1}(C, 0)$ are both points, while:
$$
\begin{array}{l}
\Hilb^{2}(C, 0) = \{(t^2 + \lambda_1t^{2k+1})\}\cup\{(t^{4}, t^{2k+1})\} \\
\Hilb^{3}(C, 0) = \{(t^{4} + \lambda_{1}t^{2k+1}, t^{2k+3})\}\cup\{(t^{6}, t^{2k+1})\} \\
\Hilb^{4}(C, 0) = \{(t^{4} + \lambda_1 t^{2k+1} + \lambda_2 t^{2k+3})\}\cup\{(t^{6} + \lambda_{1}t^{2k+1}, t^{2k+3})\}\cup\{(t^{8}, t^{2k+1})\} \\
\vdots \\
\Hilb^{m}(C, 0) = \{(t^{m} + \lambda_1t^{m+1} + \lambda_t^{m+3} + \cdots + \lambda_{k}t^{m + 2k - 1})\} \cup \\ \hspace{2.25cm} \{(t^{m+1} + \lambda_1 t^{m+2} + \cdots + \lambda_{k-1}t^{m + 2k - 2}, t^{m + 2k})\} \cup \cdots \cup \{(t^{m+k}, t^{m+k+1})\}
\end{array}
$$

\noindent for $m \geq 2k$. So we can pave $\Hilb^{m}(C,0)$ by affine spaces, and the closure of each one of these affine spaces equals the union of the affine spaces of smaller or equal dimension. Thus, we see that $\Hilb^{2}, \Hilb^{3}$ are homeomorphic to $\P^1$, $\Hilb^4, \Hilb^5$ have the homology of $\P^2$, $\dots$, $\Hilb^{m}$  have the homology of $\P^k$ for $m \geq 2k$ (and, in fact, $\Hilb^{m} \cong \Hilb^{n}$ for $m, n \geq 2k$). %With a little more work, one can show that even the homotopy types coincide in this case.

The generating function has the form
\begin{multline*}
1+q+(q^2+q^3)(1+t^2)+(q^4+q^5)(1+t^2+t^4)+\ldots+q^{2k}(1+t^2+\ldots+t^{2k})+\ldots=\frac{1+q^2t^2+q^4t^4+\ldots+q^{2k}t^{2k}}{1-q}.
\end{multline*}

\noindent Up to the change of variables $t \mapsto 1/qt^2$, this again coincides with the Poincar\'e polynomial of $\HHH^{a=0}(T(2, \linebreak 2k+1))$ computed in Example \ref{ex: t22}. 
\end{example}

In Examples \ref{ex: hopf}--\ref{ex: hilb y2=x2k+1} we have verified Conjecture \ref{conj: ORS} for the Hopf link as well as for positive $(2, 2k+1)$-torus knots. However, Conjecture \ref{conj: ORS} is wide open in general. Here we collect some facts and references for partial results:
\begin{itemize}
    \item[(a)] By a remarkable result of Maulik \cite{Maulik}, the generating function for the Euler characteristics of the Hilbert schemes matches the HOMFLY-PT polynomial:
    $$
    \sum_{k=0}^{\infty}q^k\chi(\Hilb^k(C,0))=\mathrm{HOMFLY-PT}(L;q,a=0).
    $$
    \item[(b)] For torus knots, both sides can be computed combinatorially and compared. The Khovanov-Rozansky homology is computed by Theorem \ref{th: torus knot recursions} above, while $\Hilb^k(C,0)$ has an explicit paving by affine cells. The combinatorial formula for the dimension of these cells is given in \cite{ORS} (see more details below), this yields an explicit formula for the Poincar\'e polynomial of the homology.
    \item[(c)] Recall that for an $r$-component link $L$ its homology $\HHH^0(L)$ has a natural action of the polynomial algebra $\C[x_1,\ldots,x_r]$. For $r=1$, this action on the Hilbert scheme side was constructed in \cite{MS,MY,Rennemo} and for $r>1$ it was constructed in \cite{Kivinen}. Roughly speaking, the operator $x_i$ adds a point on $i$-th component of the curve $C$, but one needs to use a versal deformation of $C$ to make it precise. We refer to \cite{Kivinen} for more details.
\end{itemize}

\subsection{Affine Springer theory}

Next, we would like to give yet another interpretation of $\Hilb^n(C,0)$ using geometric representation theory. Let us choose a projection of $C$ to some line, and let $n$ be the degree of this projection. We will regard the line as a local model for the ``base curve" and $C$ as a ``spectral curve''. 

\begin{remark}
The choice of the projection naturally splits the unit sphere in $\C^2$ as a union of two solid tori. Indeed, the equation of the sphere is $|x|^2+|y|^2=\varepsilon^2$ and the solid tori are $|x|^2\le \frac{\varepsilon^2}{2}$ and $|y|^2\le \frac{\varepsilon^2}{2}$. For $\varepsilon$ small the intersection of $C$ with a sphere defines a closed $n$-strand braid which is contained in one of the tori. This is known as a braid monodromy construction, see e. g. \cite{Artal} and references therein.
\end{remark}

We will use the following results.

\begin{lemma}\label{lem: companion matrix}
Let $C$ be a germ of an arbitrary plane curve (possibly non-reduced) given by the equation $\{f(x,y)=0\}$. 

(a) One can replace $f(x,y)$ by a polynomial of some degree $n$ in $x$ with coefficients given by power series in $y$.  

(b) A (topological) basis in $\mathcal{O}_{C,0}$ is given by monomials of the form $x^{a}y^{b}$, $a \leq n-1$. In other words, $\mathcal{O}_{C,0}$ is a free $\C[[y]]$--module of rank $n$ with basis $1,\ldots,x^{n-1}$.

(c) The  multiplication by $x$ and $y$ in this basis is given by the matrices:
$$
Y \mapsto \left(\begin{matrix} y & 0 & 0 & \cdots & 0 \\ 0 & y & 0 &  \cdots & 0 \\ 0 & 0 & y & \cdots & 0 \\ \vdots & \vdots & \vdots & \ddots & \vdots \\ 0 & 0 & 0 & \cdots & y \end{matrix} \right), \qquad X \mapsto \left( \begin{matrix} 0 & 0 & 0 & \cdots & -f_{0}(y) \\ 1 & 0 & 0 & \cdots & -f_{1}(y) \\ 0 & 1 & 0 & \cdots & -f_{2}(y) \\ \vdots & \vdots & \vdots & \ddots & \vdots \\ 0 & 0 & 0 & \cdots & -f_{n-1}(y) \end{matrix}\right)
$$
In particular, the characteristic polynomial of the second matrix equals $\det(X-x\cdot I)=f(x,y)$.
\end{lemma}

\begin{proof}
(a) We may assume that $f(0,0) = 0$. The Weierstrass preparation theorem says that in the local ring $\C[[x,y]]$ we can write $f$ as
$$
f = u(x^{n} + f_{n-1}(y)x^{n-1} + \cdots + f_{0}(y))
$$
where $u$ is a unit and $f_{n-1}(y), \dots, f_{0}(y) \in \C[[y]]$. Thus, we can write the local ring $\mathcal{O}_{C, 0} = \mathbb{C}[[x,y]]/(f)$ as $\mathbb{C}[[x,y]]/(x^n + \cdots + f_{0}(y))$. The value $n$ is the degree of the projection of the curve $f(x,y) = 0$ to the $y$-axis (i.e. the number of solutions of $f(x,y_0) = 0$ for generic $y_0$). 

(b) Thanks to part (a), we may replace $f$ by a polynomial of the form $x^n + f_{n-1}(y)x^{n-1} + \cdots + f_{0}(y)$. Thus, $\mathcal{O}_{C, 0}$ is the quotient of the algebra $\C[[x,y]]$ modulo the relation $x^{n} = -f_{n-1}(y)x^{n-1} - \cdots - f_{0}y$. It is easy to see that this is a free $\mathbb{C}[[y]]$-module with basis $1, x, x^2, \dots, x^{n-1}$. 

(c) Clearly, $y(x^{a}y^{b}) = x^{a}y^{b+1}$, while
$$
x(x^{a}y^{b}) = \begin{cases} x^{a+1}y^{b} & a < n-1 \\ (-f_{n-1}(y)x^{n-1} - \cdots - f_{0}(y))y^{b} & a = n-1 \end{cases}
$$
Which coincides with the formula in part (c). 
\end{proof}%Thinking of $\mathcal{O}_{C, 0}$ as a free $\mathbb{C}[[y]]$-module of rank $n$, the action of $x$ and $y$ are given by the following matrices (with coefficients in $\mathbb{C}[[y]]$):

%$$
%y \mapsto \left(\begin{matrix} y & 0 & 0 & \cdots & 0 \\ 0 & y & 0 &  \cdots & 0 \\ 0 & 0 & y & \cdots & 0 \\ \vdots & \vdots & \vdots & \ddots & \vdots \\ 0 & 0 & 0 & \cdots & y \end{matrix} \right), \qquad x \mapsto \left( \begin{matrix} 0 & 0 & 0 & \cdots & -f_{0}(y) \\ 1 & 0 & 0 & \cdots & -f_{1}(y) \\ 0 & 1 & 0 & \cdots & -f_{2}(y) \\ \vdots & \vdots & \vdots & \ddots & \vdots \\ 0 & 0 & 0 & \cdots & -f_{n-1}(y) \end{matrix}\right)
%$$
\begin{remark}
Note that in this presentation the roles of $x$ and $y$ are not symmetric, and the value of $n$ depends on the choice of the projection.
\end{remark}

\begin{example}\label{ex:cusp}
For the cusp $C=\{x^2=y^3\}$ we have $\CO_{C,0}=\C[[x]]\langle 1,y,y^2\rangle$ so that $$Y=\left(\begin{matrix}0 & 0 & x^2\\ 1 & 0 & 0\\ 0 & 1 & 0\\ \end{matrix}\right).
$$
On the other hand, we can choose a different projection and write 
$\CO_{C,0}=\C[[y]]\langle 1,x\rangle$ so that $$X=\left(\begin{matrix}0 & y^3\\ 1 & 0\\ \end{matrix}\right).
$$
In both cases the characteristic polynomial equals (up to sign) $x^2-y^3$.
\end{example}

We will use Lemma \ref{lem: companion matrix} to give a description of $\Hilb^{N}(C, 0)$ when $N \gg 0$ and $C$ is irreducible, see also Section \ref{sec: gasf} below. First, let us recall that for the group $SL_n$ the {\bf affine Grassmannian} is the ind-variety
$$
Gr_{SL_n} := SL_n(\C((x)))/SL_n(\C[[x]]).
$$
The affine Grassmannian $Gr_{SL_n}$ has the following interpretation. A {\bf lattice} $V \subseteq \C((x))^n=\C^n((x))$ is a free $\C[[x]]$-submodule of rank $n$ such that $V\otimes_{\C[[x]]}\C((x))=\C^n((x))$. In other words, a lattice $V$ is the $\C[[x]]$-span of a $\C((x))$-basis $(v_1, \dots, v_n)$ of $\C^n((x))$. Let us say that a lattice $V$ is of $SL_n$-type if we can find such a basis so that the determinant of the matrix with columns $v_1, \dots, v_n$ is $1$. It is known then that the affine Grassmannian parametrizes such lattices,
$$
Gr_{SL_{n}} = \left\{V \subseteq \C^n((x)) : V \; \text{is a lattice of $SL_n$-type}\right\}.
$$

\begin{remark}
Of course, one can do a similar construction with $GL_n$ instead of $SL_n$, and obtain that the affine Grassmannian $Gr_{GL_n} = GL_n(\C((x)))/GL_n(\C[[x]])$ parametrizes all lattices in $\C^n((x))$. 
\end{remark}

Using this description, if $Y$ is an $n\times n$-matrix with coefficients in $\C((x))$ we can define the {\bf affine Springer fiber}
$$
\Sp_{Y}:= \left\{ V \in Gr_{SL_n} \mid YV \subseteq V \right\} \subseteq Gr_{SL_n}.
%\left\{V\ \text{lattice in}\ \C^n((x)):xV\subset V,YV\subset V \right\}.
$$
We will be interested in the case when the matrix $Y$ comes from a polynomial $f(x,y)$ via the Weierstrass preparation theorem, as in Lemma \ref{lem: companion matrix} (with the roles of $x$ and $y$ interchanged) and Example \ref{ex:cusp}. In this case, the affine Springer fiber $\Sp_Y$ has the following properties: %Here by a lattice in $\C^n((x))$ we mean a free $\C[[x]]$ module of rank $n$. The set of all lattices in $\C^n((x))$ is known as {\bf affine Grassmannian} for $GL(n)$,so that $\Sp_{Y}$ is naturally a subset of the affine Grassmannian. 

%It has the following properties:
\begin{itemize}
    \item[(a)] If $(C,0)$ is irreducible then $\Sp_{Y}$ is isomorphic to the compactified Jacobian of $C$, that is, the moduli space of rank 1 torsion free sheaves ''of degree zero" on $C$ (see e.g. \cite{MY}). It is also isomorphic to the Hilbert scheme $\Hilb^N(C,0)$ for $N\gg 0$. In particular, $\Sp_Y$ is a projective variety.
    \item[(b)] If $(C,0)$ has $r$ components then there is an action of $\Z^{r-1}$ on $\Sp_{Y}$ by translations, and of $(\C^*)^{r-1}$. In particular, $\Sp_{Y}$ is an ind-variety with infinitely many irreducible components, all of the same dimension, which are permuted by the action of $\Z^{r-1}$.
 \end{itemize}
 \begin{remark}
 In the $GL_n$-case we drop the degree condition in part (a) above, and in (b) we get an action of $\Z^n$. The $GL_n$-affine Springer fibers can be interpreted as the compactified Picard schemes of $C$, and in this generality are simply unions of $\Z=\pi_1(GL_n)$ copies of the $SL_n$-affine Springer fibers.
 \end{remark}
\begin{remark}
\label{rem: jacobian}
Alternatively, one can define compactified Jacobian of $(C,0)$ by considering a parameterization of the curve $(x(t),y(t))$ so that $\CO_{C,0}=\C[[x(t),y(t)]]$. In this case the compactified Jacobian is the moduli space of $\CO_{C,0}$-submodules of $\C[[t]]$ up to a shift by a power of $t$.
\end{remark}

If $C$ is irreducible and reduced, there is a deep cohomological relationship between the affine Springer fiber/compactified Jacobian and the Hilbert schemes on $(C,0)$, coming from the natural Abel-Jacobi map interpreting ideal sheaves as torsion-free sheaves.

\begin{theorem}[\cite{MY,MS}]
\label{th: perverse}
One has
$$
\bigoplus_{k=0}^{\infty}H^*(\Hilb^k(C,0))=\gr_{P}H^*(\Sp_{Y})\otimes \C[x],
$$
where $\gr_P$ refers to the associated graded with respect to a certain ``perverse" filtration on the cohomology of $\Sp_{Y}$.

Furthermore, there is an action of $\mathfrak{sl}_2$ on $H^*(\Sp_{Y})$ satisfying ``curious hard Lefshetz" property with respect to the perverse filtration. 
\end{theorem}
\begin{remark}
A slightly weaker version of the Theorem also holds for the reducible case with appropriate modifications, as shown in \cite[Theorem 3.11]{MY}. A representation-theoretic proof in the irreducible case is given in \cite{Rennemo}.
\end{remark}
 
The action of $\mathfrak{sl}_2$ is similar to the action on the cohomology of the braid variety (with weight filtration) from Theorem \ref{th: sl2 braid variety} and to the action in link homology from Theorem \ref{th: tautological}. 

\begin{conjecture}[Shende \cite{STZ}]
\label{conj: shende}
Let $C$ be an irreducible plane curve singularity, $L$ the corresponding algebraic knot and $\beta$ the corresponding braid on $n$ strands. Then one has the isomorphism of the compactly supported cohomology of the braid variety and the singular cohomology of the affine Springer fiber
$$
H_c^*(X(\beta\Delta;w_0)/(\C^*)^{n-1})\simeq H^*(\Sp_{Y})
$$
where the (halved) weight filtration on the left hand side matches the perverse filtration on the right.
\end{conjecture}

\begin{remark}
Conjecture \ref{conj: shende} is closely related to the framework of so-called $P=W$ conjectures of de Cataldo-Hausel-Migliorini \cite{dCHM} relating the weight filtration on the cohomology of the character varieties and their cousins (such as braid varieties) and the perverse filtration on the cohomology of the Hitchin moduli spaces and their cousins (such as affine Springer fibres). We refer to \cite{dCHM} for more context.
\end{remark}

\begin{example}
Let $C=\{x^2=y^2\}$, then
$$
X=\left(\begin{matrix}0 & y^2\\ 1 & 0\\ \end{matrix}\right).
$$
The corresponding affine Springer fiber is an infinite chain of $\P^1$, with the lattice $\Z$ acting by translations.
\end{example}

\begin{example}
Let us compute the homology of the compactified Jacobian of the singularity $\{x^{m} = y^{n}\}$. As in Remark \ref{rem: jacobian} we are classifying $\C[[t^n,t^m]]$-submodules $M \subseteq \C[[t]]$. Any element of $\C[[t]]$ has an {\em order}, that is, minimal degree in $t$ with a nonzero coefficient. Up to a shift by a power of $t$, we can assume that such a module $M$ contains an element of order 0 in $t$. Let $\Gamma_{m,n}$ denote the semigroup generated by $m$ and $n$, then for each element of $\Gamma_{m,n}$ there is a corresponding element of $M$ and it is completely determined by the element of order 0. Also, $\Gamma_{m,n}$ contains all integers starting from $(m-1)(n-1)$, hence $t^{(m-1)(n-1)}\C[[t]]\subset M$.
We have the following cases:

(a) $(m,n)=(2,3)$. We have $\Gamma_{2,3}=\{0,2,3,4,\ldots\}$ and
there are two types of $\C[[t^2,t^3]]$--modules:
$$
(1+\lambda t),\ (1,t)=\C[[t]].
$$
The first family of modules forms an affine line, so altogether we get $\overline{JC_{2,3}}=\P^1$.

(b) $(m,n)=(2,2k+1)$. We have 
$$
\Gamma_{2,2k+1}=\{0,2,\ldots,2k,2k+1,2k+2,\ldots\}
$$
and
there are the following types of $\C[[t^2,t^{2k+1}]]$--modules:
$$
 \begin{array}{l}
(1+\lambda_1 t+\ldots+\lambda_{k}t^{2k-1}),\\
(1+\lambda_1 t+\ldots+\lambda_{k-1}t^{2k-3},t^{k-1}),\\
\vdots \\
(1+\lambda_1 t, t^3,\ldots,t^{k-1}),\\
(1,t,t^3,\ldots,t^{k-1})=\C[[t]].
\end{array}
$$
This yields a cell decomposition of the compactified Jacobian with one cell of dimensions $k,k-1,\ldots,1,0$. The reader should compare this with Example \ref{ex: hilb y2=x2k+1}.

(c) $(m,n)=(3,4)$. We have 
$$
\Gamma_{3,4}=\{0,3,4,6,7,\ldots\}
$$
and
there are the following types of $\C[[t^3,t^{4}]]$--modules:
$$
\begin{array}{l}
(1+\lambda_1t+\lambda_2t^2+\lambda_3t^5),\\
 (1+\lambda_1t+\lambda_2t^2,t^5),\\
 (1+\lambda_1t,t^2,t^5),\\
 (1+\lambda_1t+\lambda_2t^2,t+\mu t^2, t^5),\\
 (1,t,t^2)=\C[[t]].
\end{array}
$$
In the fourth case we can change basis to
$$
(1+(\lambda_2-\mu\lambda_1)t^2,t+\mu t^2, t^5),
$$
so we can assume $\lambda_1=0$ and there are two parameters $\mu,\lambda_2$. The compactified Jacobian then has one 3-cell (first case), two 2-cells (second and fourth cases), one 1-cell and one 0-cell. It is a singular 3-dimensional variety with homology given by the following table 
\begin{center}
\begin{tabular}{|c|c|c|c|c|c|c|c|}
\hline
  & $H^0$ & $H^1$ & $H^2$ & $H^3$ & $H^4$ & $H^5$ & $H^6$ \\
\hline
     $k-p=0$ & 1 & 0 & 1 & 0 & 1 & 0 & 1 \\
\hline
     $k-p=1$ &   0 & 0 & 0 & 0 & 1 & 0 & 0 \\
\hline    
\end{tabular}
\end{center}
The rows indicate the difference between the homological degree and the perverse filtration, so that $H^4$ has rank 2 and nontrivial perverse filtration. This table matches the one in 
Example \ref{ex: E6 homology}.
\end{example}

For general coprime $(m,n)$ the compactified Jacobian of $C=\{x^m=y^n\}$ is always paved by affine cells, and the dimensions of these cells can be computed by a combinatorial formula which can be written in several different ways \cite{Hikita,GMV,GM1,GM2,LuSm,Piontkowski}. We will not write an explicit formula but mention the following:

\begin{lemma}
The Euler characteristic of the compactified Jacobian of $C=\{x^m=y^n\}$ equals the {\bf rational Catalan number}:
$$
c_{m,n}=\frac{(m+n-1)!}{m!n!}=\frac{1}{m+n}\binom{m+n}{n}.
$$
\end{lemma}

\begin{proof}
The natural $\C^*$-action on the curve lifts to the $\C^*$ action on the Hilbert scheme of points and on the compactified Jacobian. It is easy to see that (if $m,n$ are coprime) the only fixed points of this action are ideals (resp. submodules) generated by monomials in $t$. The fixed points in the compactified Jacobian then correspond (up to translation by an integer) to the subsets $\Delta\subset \Z_{\ge 0}$ which are invariant under both shifts by $m$ and $n$, that is, $\Delta+m\subset \Delta,\Delta+n\subset \Delta$. There are several ways to enumerate such subsets, here is one.

Since we are considering invariant subsets of $\Z_{\geq 0}$ up to translation by an integer, we may assume that $0\in \Delta$, so that the whole semigroup $\Gamma_{m,n}$ is contained in $\Delta$. Consider the $m\times n$ rectangle and fill it with numbers $mn-mx-ny$ where we assume that the bottom left box has coordinates $(x,y)=(1,1)$. Here is an example for $(m,n)=(3,4)$:

\begin{center}
\begin{tikzpicture}
\draw (0,0)--(0,3)--(4,3)--(4,0)--(0,0);
\draw (0,1)--(4,1);
\draw (0,2)--(4,2);
\draw (0,3)--(4,3);
\draw (1,0)--(1,3);
\draw (2,0)--(2,3);
\draw (3,0)--(3,3);
\draw (0.5,0.5) node {$5$};
\draw (1.5,0.5) node {$2$};
\draw (2.5,0.5) node {$-1$};
\draw (3.5,0.5) node {$-4$};
\draw (0.5,1.5) node {$1$};
\draw (1.5,1.5) node {$-2$};
\draw (2.5,1.5) node {$-5$};
\draw (3.5,1.5) node {$-8$};
\draw (0.5,2.5) node {$-3$};
\draw (1.5,2.5) node {$-6$};
\draw (2.5,2.5) node {$-9$};
\draw (3.5,2.5) node {$-12$};
\draw [dotted] (0,3)--(4,0);
\end{tikzpicture}
\end{center}

One can check that all nonnegative integers in $\Z_{\ge 0}\setminus \Gamma_{m,n}$ appear exactly once in this rectangle in the cells strictly below the diagonal. It follows that the $(m,n)$-invariant subsets $\Delta$ containing $0$ correspond to Young diagrams in this rectangle contained strictly below the diagonal, or, equivalently, the lattice paths from northwest to southeast corner which stay below the diagonal. Here is an example for the subset $\Delta=\{0,{\bf 1},3,4,{\bf 5},6,\ldots\}$ where we mark in bold the numbers added to $\Gamma_{3,4}$:

\begin{center}
\begin{tikzpicture}
\draw (0,0)--(0,3)--(4,3)--(4,0)--(0,0);
\draw (0,1)--(4,1);
\draw (0,2)--(4,2);
\draw (0,3)--(4,3);
\draw (1,0)--(1,3);
\draw (2,0)--(2,3);
\draw (3,0)--(3,3);
\draw (0.5,0.5) node {$\mathbf{5}$};
\draw (1.5,0.5) node {$2$};
\draw (2.5,0.5) node {$-1$};
\draw (3.5,0.5) node {$-4$};
\draw (0.5,1.5) node {$\mathbf{1}$};
\draw (1.5,1.5) node {$-2$};
\draw (2.5,1.5) node {$-5$};
\draw (3.5,1.5) node {$-8$};
\draw (0.5,2.5) node {$-3$};
\draw (1.5,2.5) node {$-6$};
\draw (2.5,2.5) node {$-9$};
\draw (3.5,2.5) node {$-12$};
\draw [dotted] (0,3)--(4,0);
\draw [line width=5] (0,3)--(0,2)--(1,2)--(1,0)--(4,0);
\end{tikzpicture}
\end{center}
The number of such paths equals 
$$
c_{m,n}=\frac{(m+n-1)!}{m!n!}=\frac{1}{m+n}\binom{m+n}{n}.
$$
\end{proof}

\begin{theorem}[\cite{GMV2}]
The dimensions of the cells in the compactified Jacobians of $\{x^m=y^n\},GCD(m,n)=1$ are given by a recursive formula, which matches the recursion for Khovanov-Rozansky homology of torus knots in Theorem \ref{th: torus knot recursions} and in \cite{HM}.
\end{theorem}

We refer to \cite{GMV2} for more details. Also, we have the following description of the cohomology ring of the compactified Jacobian in this case.

\begin{theorem}[\cite{OY,OY2}]
\label{th: OY}
The cohomology ring of the compactified Jacobian of $\{x^m=y^n\},GCD(m,n)=1$ is generated by tautological classes of degrees $2,4,\ldots,2n-2$ satisfying certain explicit relations \cite{OY2}.
\end{theorem}

\begin{remark}
It is natural to expect that the tautological classes in Theorem \ref{th: OY} are related to the ones in link homology, defined in Theorem \ref{th: tautological}. The precise relation between these is yet to be established.
\end{remark}

Next, we consider curves with more components. Consider the curve $\{x^{kn}=y^n\}$ which corresponds to the $(n,kn)$ torus link. The companion matrix of the polynomial $x^{kn} - y^n$ is conjugate in $GL(n, \C((t)))$, or even $SL_n(\C((t)))$, to the matrix 
$$
Y=\left(
\begin{matrix}
\zeta_1x^n & \cdots & 0\\
\vdots & \ddots & \vdots\\
0 & \cdots & \zeta_nx^n\\
\end{matrix}
\right),
$$
where $\zeta_1,\ldots,\zeta_n$ are distinct roots of unity of degree $n$.  There is an action of the lattice $\Z^{n-1}$ on $\Sp_{Y}$ by translations, as well as of the diagonal torus $(\C^*)^{n-1}$, both coming from the natural centralizer action on $\Sp_{Y}$. In \cite{Kivinen2} the second author compared the equivariant Borel-Moore homology of $\Sp_{Y}$ in this case with the homology of $(n,kn)$ torus link following Theorem \ref{th: full twist}. 

\begin{theorem}[\cite{Kivinen2}]
(a) One has
$$
H_{*, BM}(\Sp_Y)=\HHH^0(T(n,kn))\bigotimes_{\C[x_1,\ldots,x_n]/(\prod_i x_i-1)}\C[x_1^{\pm},\ldots,x_n^{\pm}]/\left(\prod_i x_i -1\right)
$$
where the action of $x_i$ on the left hand side is given by the lattice $\Z^{n-1}$, and on the right hand side by Theorem \ref{th: full twist}.

(b) Similarly,
$$
H^{T}_{*, BM}(\Sp_Y)=\HY^0(T(n,kn))\bigotimes_{\C[x_1,\ldots,x_n]/\prod_i x_i-1}\C[x_1^{\pm},\ldots,x_n^{\pm}]/\left(\prod_i x_i-1\right)
$$
where $T=(\C^*)^{n-1}$ and the equivariant parameters $y_1,\ldots,y_n$ with $\sum_i y_i=0$ match the ones appearing in the $y$-ification on the right. One can avoid the restrictions to the codimension 1 subtori by considering the $GL_n$-affine Springer fibers instead.
\end{theorem}

For more details on the geometry of the affine Springer fibers corresponding to $(n,kn)$ torus links, and its connections to quantum groups we refer the reader to \cite{BBSV,Pablo,GKM}. The braid variety for  $(n,kn)$ torus links is discussed in \cite{HMW, Boalch}.

\subsection{Generalized affine Springer fibers}\label{sec: gasf}

In \cite{BFN,BFN2} Braverman, Finkelberg and Nakajima defined a remarkable family of algebras, called Coulomb branch algebras, associated to a reductive group $G$ and its representation $N$. In fact, they defined both a commutative algebra $\CA_{G,N}$ and its %filtered
quantization $\CA^{\hbar}_{G,N}$. In \cite{GKMPurity} Goresky-Kottwitz-MacPherson defined a {\bf generalized affine Springer fiber} for $(G,N)$ parametrized by a vector in $N((t))$ and in \cite{HKW, GarKiv} Hilburn-Kamnitzer-Weekes and Garner-Kivinen proved (under some mild assumptions), that $\CA_{G,N}$ acts on its Borel-Moore homology while $\CA^{\hbar}_{G,N}$ acts on its loop rotation-equivariant Borel-Moore homology. 

It turns out that {\bf all} Hilbert schemes of points on singular curves are special cases of this construction. 

\begin{theorem}[\cite{GarKiv}]
\label{th: generalized Springer}
 (a) For any curve $C$ (not necessarily reduced!) the union $\sqcup_{k}\Hilb^k(C,0)$ is isomorphic to a certain generalized affine Springer fiber for $G=GL(n),N=\mathfrak{gl}(n)\oplus \C^n$.
 Here $n$, as above, denotes the degree of the projection of $C$ onto some line. 
 
 (b) There is an action of $\CA_{G,N}$  on $\bigoplus_{k}H^*(\Hilb^k(C,0))$
 
 (c) If the curve $C$ is quasi-homogeneous then there is a $\C^*$-action on $\Hilb^k(C,0)$ and $\CA^{\hbar}_{G,N}$ acts in the corresponding equivariant cohomology.
\end{theorem}
 
\begin{remark}
\label{rem: Kodera Nakajima}
 By the work of Kodera and Nakajima \cite{KN}, the quantized BFN algebra $\CA^{\hbar}_{G,N}$ for 
 $$
 (G,N)=(GL(n),\mathfrak{gl}(n)\oplus \C^n)
 $$
 is isomorphic to the {\bf spherical rational Cherednik algebra} of type $S_n$. We refer to \cite{GSV} for more details and the combinatorial description of the action of the rational Cherednik algebra on  $\bigoplus_{k}H^*_{\C^*}(\Hilb^k(C,0))$.
\end{remark}

\section{$\Hilb^n(\C^2)$ and link homology}

\subsection{Hilbert scheme and its properties}

In this lecture we outline some results and conjectures relating the Hilbert scheme of points on the plane to link homology. First, we recall the definition and main properties of the Hilbert scheme, and refer to the book \cite{NakBook} for all details. 

The Hilbert scheme is defined as
$$
\Hilb^n(\C^2):=\{I\subset \C[x,y]\ \text{ideal}:\dim \C[x,y]/I=n\}.
$$
It is a smooth algebraic symplectic variety of dimension $2n$. It is also a {\bf conical symplectic resolution} in the sense of \cite{BPW,BLPW}: the natural map $\pi:\Hilb^n(\C^2)\to S^n(C^2)$ which sends an ideal to its support is a resolution of singularities, and
$$
S^n\C^2=\Spec\ \C[x_1,\ldots,x_n,y_1,\ldots,y_n]^{S_n}
$$
is a singular affine Poisson variety in a way that is compatible with the map $\pi$. The two $\C^*$ actions on $\C^2$ which scale the coordinates lift to an action of $(\C^*)^2$ on the Hilbert scheme. Inside $(\C^*)^2$, we have two distinguished one-dimensional tori. 

\begin{itemize}
    \item The Hamiltonian torus $H = \{(h, h^{-1}) \mid h \in \C^*\}$ acts as $(x,y)\mapsto (hx,h^{-1}y)$. The reason it is called Hamiltonian is that it preserves the symplectic form.
    
    \item The scaling torus $S = \{(s, s) \mid s \in \C^*\}$ acts as $(x,y)\mapsto (sx,sy)$. Note that it scales the symplectic form by $s^2$. It contracts the whole $S^n\C^2$ to the origin.
\end{itemize}

\begin{remark}
Abusing the notation, let us denote by $s$ and $t$ the gradings on a $(\C^{*} \times \C^{*})$-equivariant sheaf on $\Hilb^n(\C^2)$ induced by the actions of $S$ and $T$, respectively. Likewise, we denote by $q$ and $t$ the gradings corresponding to the left and right, respectively, $\C^*$-factors in $\C^* \times \C^*$. Note that we have $h = qt^{-1}$ and $s = qt$. 
\end{remark}

The action of the torus $S$ is the reason why we have the adjective {\bf conical} in conical symplectic resolution, for it contracts the affinization $S^n(\C^2)$ of $\Hilb^n(\C^2)$ to a single point. The Hamiltonian torus $H$ has the following properties.

\begin{itemize}
    \item The attracting subvariety 
    $$
    L=\{p\in \Hilb^n(\C^2): \lim_{h\to 0}h.p\ \text{exists}\}
    $$
    is Lagrangian in $\Hilb^n(\C^2)$ and coincides with $\Hilb^n(\C^2,\C)=\pi^{-1}(\{y=0\})$ (since $\lim_{h\to 0}(hx,h^{-1}y)$ exists if and only if $y=0$).
    \item The $H$-fixed points are isolated and correspond to monomial ideals. The monomial ideals in $\Hilb^n(\C^2)$ are in correspondence with partitions of $n$, as follows. Given a Young diagram $\lambda$ of size $n$ (in French notation), the monomial ideal $I_{\lambda}$ is generated by all monomials outside $\lambda$.
    \begin{center}
        \begin{tikzpicture}
        \filldraw [lightgray] (0,7)--(0,4)--(1,4)--(1,3)--(2,3)--(2,2)--(5,2)--(5,1)--(6,1)--(6,0)--(9,0)--(9,7)--(0,7);
            \draw (0,0)--(9,0);
            \draw (0,1)--(9,1);
            \draw (0,2)--(9,2);
            \draw (0,3)--(9,3);
            \draw (0,4)--(9,4);
            \draw (0,5)--(9,5);
            \draw (0,6)--(9,6);
            \draw (0,7)--(9,7);
            \draw (0,0)--(0,7);
            \draw (1,0)--(1,7);
            \draw (2,0)--(2,7);
            \draw (3,0)--(3,7);
            \draw (4,0)--(4,7);
            \draw (5,0)--(5,7);
            \draw (6,0)--(6,7);
            \draw (7,0)--(7,7);
            \draw (8,0)--(8,7);
            \draw (9,0)--(9,7);
            \draw [line width=3] (0,7)--(0,4)--(1,4)--(1,3)--(2,3)--(2,2)--(5,2)--(5,1)--(6,1)--(6,0)--(9,0);
            
            \draw (0.5,4.5) node {$y^4$};
            \draw (0.5,3.5) node {$y^3$};
            \draw (1.5,3.5) node {$xy^3$};
            \draw (0.5,2.5) node {$y^2$};
            \draw (1.5,2.5) node {$xy^2$};
            \draw (2.5,2.5) node {$x^2y^2$};
             \draw (0.5,1.5) node {$y$};
             \draw (1.5,1.5) node {$xy$};
             \draw (2.5,1.5) node {$x^2y$};
             \draw (3.5,1.5) node {$x^3y$};
             \draw (4.5,1.5) node {$x^4y$};
            \draw (5.5,1.5) node {$x^5y$};
            \draw (0.5,0.5) node {$1$};
            \draw (1.5,0.5) node {$x$};
            \draw (2.5,0.5) node {$x^2$};
            \draw (3.5,0.5) node {$x^3$};
            \draw (4.5,0.5) node {$x^4$};
            \draw (5.5,0.5) node {$x^5$};
            \draw (6.5,0.5) node {$x^6$};
        \end{tikzpicture}
    \end{center}
    For example, the diagram $\lambda=(6,5,2,1)$ corresponds to
    the ideal $I_{\lambda}=(x^6,x^5y,x^2y^2,xy^3,y^4)$.
\end{itemize}

Note that there is a tautological rank $n$ bundle $\CT$ on the Hilbert scheme, whose fiber over an ideal $I$ given by $\C[x,y]/I$. We will need a line bundle
    $$
    \det \CT=\wedge^n \CT =: \CO(1).
    $$
    
     Finally, to connect to Section \ref{sec: gasf}, in the terminology of \cite{BFN,BFN2}, $\Hilb^n(\C^2)$ is {\bf both} the Higgs branch and the Coulomb branch for the 3d $N=4$ gauge theory corresponding to $(G,N)=(GL(n),\mathfrak{gl}(n)\oplus \C^n)$ (compare with Theorem \ref{th: generalized Springer} and Remark \ref{rem: Kodera Nakajima} above) and to the quiver
    
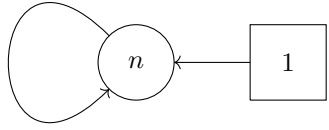
\begin{center}
        \begin{tikzpicture}
            \draw (0,0) circle (0.5);
            \draw (1.5,-0.5)--(1.5,0.5)--(2.5,0.5)--(2.5,-0.5)--(1.5,-0.5);
            \draw (0,0) node {$n$};
            \draw (2,0) node {$1$};
            \draw[->] ({-1/(2*sqrt(2))},{1/(2*sqrt(2))})..controls ({-3/sqrt(2)},{3/sqrt(2)}) and ({-3/sqrt(2)},{-3/sqrt(2)})..({-1/(2*sqrt(2))},{-1/(2*sqrt(2))});
            \draw[->] (1.5,0)--(0.5,0);
        \end{tikzpicture}
    \end{center}
    of affine type $\widehat{A_1}$.

\subsection{Hilbert schemes and link homology}

We are ready to discuss the relation between the Hilbert scheme and link homology. The following conjecture was formulated in \cite{GNR} and mostly proved in a series of papers by Oblomkov and Rozansky \cite{OR1,OR2,OR3,OR4,OR5,OR6,OR7,OR8}.

\begin{conjecture}
\label{conj: GNR}
To a braid $\beta$ on $n$ strands one can associate a $\C^*\times \C^*$-equivariant coherent sheaf $\CF_{\beta}$ on $\Hilb^n(\C^2,\C)$ with the following properties:
\begin{itemize}
    \item[(a)] One has
    $$
    \HHH(\beta)\simeq H^*_{\C^*\times \C^*}\left(\Hilb^n(\C^2,\C),\CF_{\beta}\otimes \wedge^{\bullet}\CT^{\vee}\right)
    $$
    as triply graded vector spaces. The $q$ and $t$ gradings on the right hand side correspond to the $\C^*\times \C^*$ action, and the $a$-grading corresponds to the power of $\wedge^{\bullet}\CT^{\vee}$. The relation between the gradings $(a,q,t)$ on the right and $(A,Q,T)$ on the left is given by \eqref{eq: change of variables}.
    \item[(b)] The action of symmetric functions in $x_i$ on the left corresponds to the action of $$\C[x_1,\ldots,x_n]^{S_n}\subset \C[x_1,\ldots,x_n,y_1,\ldots,y_n]^{S_n}=H^0(\Hilb^n(\C^2),\CO)$$
    on the right hand side. In particular, the action of $x_i$ on $\HHH(\beta)$ determines the support of $\CF_{\beta}$.
    \item[(c)] In particular, (b) implies that for $\beta$ which closes up to a knot all $x_i$ act the same way, and $\CF_{\beta}$ is supported on $\Hilb^n(\C^2,0)\times \C\subset \Hilb^n(\C^2,\C)$.
    \item[(d)] Adding a full twist $\FT=(\sigma_1\cdots \sigma_{n-1})^{n}$ to a braid $\beta$ corresponds to tensoring the sheaf $\CF_{\beta}$ by $\CO(1)$.
    \item[(e)] The sheaf $\CF_{\beta}$ extends to a sheaf $\widetilde{\CF_{\beta}}$ on the whole $\Hilb^n(\C^2)$ which corresponds to the $y$-ified homology $\HY(\beta)$.
\end{itemize}
\end{conjecture}

\begin{example}
The torus braid $\beta=T(n,kn+1)$ corresponds to the line bundle $\CO(k)$ on the punctual Hilbert scheme $\Hilb^n(\C^2,0)$. In particular, $T(2,3)$ corresponds to $\CO(1)$ on $\Hilb^2(\C^2,0)=\P^1$,  $T(2,2k+1)$ corresponds to $\CO(k)$ on $\Hilb^2(\C^2,0)$, and it is easy to see that for $k>0$ the bigraded space of sections of $\CO(k)$ on $\P^1$ matches $\HHH^0(T(2,2k+1))$ computed in Example \ref{ex: two strand homology} up to regrading.

Furthermore, $T(3,4)$ corresponds to $\CO(1)$ on $\Hilb^3(\C^2,0)$ which is isomorphic to the (projective) cone over twisted cubic in $\P^3$. One can check that 
$$
\dim H^0(\Hilb^3(\C^2,0),\CO(1))=5.
$$
More generally, by \cite{HaimanCatalan} one can compute $H^0(\Hilb^n(\C^2,0),\CO(1))$ and match it with the $q,t$-Catalan numbers \cite{GaHa}.  
\end{example}

For general torus braids $\beta=T(m,n)$ the description of $\CF_{\beta}$ is more complicated. Consider the {\bf flag Hilbert scheme} 
$$
\FHilb^n(\C^2,0):=\{\C[x,y]\supset I_1\supset\ldots\supset I_n\},
$$
where all $I_k$ are ideals in $\C[x,y]$ of codimension $k$ supported at the origin. It is a very singular space, which can be equipped with the structure of a virtual complete intersection (or of a dg scheme) \cite{GNR}. It comes with the projection
$$
p:\FHilb^n(\C^2,0)\to \Hilb^n(\C^2,0),\ (I_1,\ldots,I_n)\mapsto I_n
$$
and a collection of line bundles $\CL_k=I_{k-1}/I_k$. 

\begin{theorem}[\cite{GN,GNR,OR3}]
\label{thm: refined}
Suppose that $GCD(m,n)=1$. The sheaf $\CF_{m,n}$ corresponding to $\beta=T(m,n)$ is given by
$$
\CF_{m,n}=p_*(\CL_1^{a_1}\cdots \CL_n^{a_n}),
$$
where $a_k=\lceil\frac{km}{n}\rceil-\lceil\frac{(k-1)m}{n}\rceil$.
\end{theorem}
Note that here we consider $p_*$ as the derived pushforward for the morphism of dg schemes. We refer the reader to \cite{GN,GNR,OR3} for more details.

\begin{remark}
One can consider the sheaves $p_*(\CL_1^{a_1}\cdots \CL_n^{a_n})$ for arbitrary exponents $a_1,\ldots,a_n$. These correspond to the braids $$\beta(a_1,\ldots,a_n):=\ell_1^{a_1}\cdots \ell_n^{a_n}\sigma_1\cdots \sigma_n
$$ 
where 
$$
\ell_i=\sigma_{i-1}\cdots \sigma_1\sigma_1\cdots \sigma_{i-1}
$$
are Jucys-Murphy braids. We refer to \cite{GNR,OR3} for more details on the relation between such sheaves and braids (called Coxeter braids in \cite{OR3}) and to \cite{GHSR,BHMPS} for the corresponding combinatorial results and identities. 
\end{remark}

\subsection{Procesi bundle and the identity braid}

Next, we would like to associate sheaves to the identity braid and the powers of the full twist. For this, we will need the following constructions of Haiman \cite{Haiman,Haiman2}. Let $X_n$ denote the reduced fiber product of $\Hilb^n(\C^2)$ with $(\C^2)^n$ over $S^n(\C^2)$:
\begin{center}
\begin{tikzcd}
X_n\arrow{r}\arrow{d}{q} & (\C^2)^n \arrow{d}\\
\Hilb^n(\C^2) \arrow{r}{\pi} & S^n(\C^2)\\
\end{tikzcd}
\end{center}

\begin{theorem}[\cite{Haiman}]
\label{th: isospectral}
The space $X_n$ satisfies the following properties:
\begin{itemize}
    \item[(a)] $X_n$ is the blowup of $(\C^2)^n$ along the the union of all diagonals $\{P_i=P_j\}$.
    \item[(b)] $X_n=\Proj \bigoplus_{k=0}^{\infty}J^k$ where
    $$
    J=\cap_{i\neq j} (x_i-x_j,y_i-y_j)
    $$
    is the ideal defining the union of diagonals.
    Also, $J^k$ is free over $\C[y_1,\ldots,y_n]$.
    \item[(c)] $q_*\CO_{X_n}=\CP$ is a vector bundle on $\Hilb^n(\C^2)$ of rank $n!$. It is called the {\bf Procesi bundle.}
    \item[(d)] We have the following diagram:
    
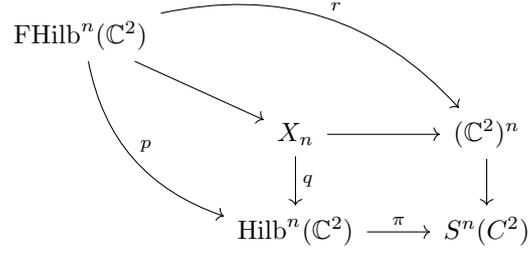
\begin{center}
        \begin{tikzcd}
            \FHilb^n(\C^2) \arrow[bend right]{ddr}{p} \arrow{dr} \arrow[bend left]{drr}{r}& & \\
            & X_n \arrow{d}{q}\arrow{r} & (\C^2)^n \arrow{d}\\
            & \Hilb^n(\C^2) \arrow{r}{\pi} & S^n(C^2)\\
        \end{tikzcd}
    \end{center}
    where the maps $p$ and $q$ are as above, and $r$ sends a flag of ideals $(I_1,\ldots,I_n)$ to the ordered collection of supports of $I_{k-1}/I_k$. Then
    $$
    p_*\CO_{\FHilb^n(\C^2)}=q_*\CO_{X_n}=\CP.
    $$
\end{itemize}
\end{theorem}
Part (d) of the theorem is implicit in \cite{Haiman} and proved in \cite{GNR} in more detail.

Finally, the sheaf $\CF_1$ corresponding to the identity braid coincides with $p_*\CO_{\FHilb^n(\C^2,\C)}$ (similarly to Theorem \ref{thm: refined}) and hence is isomorphic to the Procesi bundle $\CP$ restricted to $\Hilb^n(\C^2,\C)$. Similarly, $\beta=T(n,kn)$ correspond to the vector bundle
$$
\CF_{n,kn}=\CP\otimes \CO(k).
$$
restricted to $\Hilb^n(\C^2,\C)$. Note that for $k\ge 0$ by the projection formula we have
$$
H^0(\Hilb^n(\C^2,\C),\CP\otimes \CO(k))=H^0(X_n(\C^2,\C),\CO(k))=J^k/(y)J^k
$$
since $H^0(X_n,\CO(k))=J^k$ is free over $\C[y_1,\ldots,y_n]$.
This agrees with the Khovanov-Rozansky homology of the full twist by Theorem \ref{th: full twist}.

\subsection{$y$-ification and symmetry}

Many interesting structures in link homology become transparent when considering $\Hilb^n(\C^2)$. As we mentioned above, $y$-ification corresponds to the extension of the sheaf $\CF_{\beta}$ from $\Hilb^n(\C^2,\C)$ to a sheaf $\widetilde{\CF_{\beta}}$ on the whole Hilbert scheme of points. More precisely, we have the following $y$-ified version of Conjecture \ref{conj: GNR}:

\begin{conjecture}
\label{conj: yified GNR}
To a braid $\beta$ on $n$ strands one can associate a $\C^*\times \C^*$-equivariant coherent sheaf $\widetilde{\CF_{\beta}}$ on $\Hilb^n(\C^2)$ with the following properties:
\begin{itemize}
    \item[(a)] One has
    $$
    \HY(\beta)\simeq H^*_{\C^*\times \C^*}\left(\Hilb^n(\C^2),\widetilde{\CF_{\beta}}\otimes \wedge^{\bullet}\CT^{\vee}\right)
    $$
    as triply graded vector spaces (with the same grading conventions as in Conjecture \ref{conj: GNR}).
    \item[(b)] The action of symmetric functions in $x_i$ and $y_u$ on the left corresponds to the action of $$ \C[x_1,\ldots,x_n,y_1,\ldots,y_n]^{S_n}=H^0(\Hilb^n(\C^2),\CO)$$
    on the right hand side. 
    \item[(c)] In particular, (b) implies that for $\beta$ which closes up to a knot the sheaf $\widetilde{\CF_{\beta}}$ is supported on $\Hilb^n(\C^2,0)\times \C^2\subset \Hilb^n(\C^2)$, and agrees with the trivial extension of $\CF_{\beta}$
    \item[(d)] Adding a full twist $\FT=(\sigma_1\cdots \sigma_{n-1})^{n}$ to a braid $\beta$ corresponds to tensoring the sheaf $\widetilde{\CF_{\beta}}$ by $\CO(1)$.
    \item[(e)] The restriction of  $\widetilde{\CF_{\beta}}$ to $\Hilb^n(\C^2,\C)$ agrees with $\CF_{\beta}$.
\end{itemize}
\end{conjecture}

For example, for the identity braid we get the Procesi bundle $\widetilde{\CF_1}=\CP$. Similarly, $\beta=T(n,kn)$ correspond to the vector bundle
$$
\widetilde{\CF_{n,kn}}=\CP\otimes \CO(k).
$$
As above, by Theorem \ref{th: isospectral} the space of sections of this bundle agree with $\HY^0(T(n,kn))$ computed in Theorem \ref{th: full twist}.

Finally, we can comment on the $q-t$ symmetry and $\mathfrak{sl}_2$ action in link homology which we saw in Theorems \ref{th: tautological}, \ref{th: sl2 braid variety} and \ref{th: perverse}. The group $SL(2)\subset GL(2)$ naturally acts on $\C^2$ by linear changes of coordinates, and this action extends to the action of $SL(2)$ on $\Hilb^n(\C^2)$.

One then expects that the sheaf $\widetilde{\CF_{\beta}}$ is equivariant with respect to this action. Since $\wedge^{\bullet}\CT^{\vee}$ is also $SL(2)$-equivariant, we get a natural action of the group $SL(2)$ and its Lie algebra $\mathfrak{sl}_2$ on the corresponding link cohomology. Since the action extends to the action of $GL(2)$, it is easy to see that the generators of $\mathfrak{sl}_2$ interact with $q$ and $t$ gradings (corresponding to $\C^*\times \C^*\subset GL(2)$) correctly. The symmetry between $q$ and $t$ is then realized by the action of the Weyl group $S_2\subset SL(2)$. 

Note that $\Hilb^n(\C^2,\C)$ is not preserved by the action of $SL(2)$, so one cannot expect the action of $SL(2)$ in $\HHH(\beta)$ for arbitrary links. On the other hand, $\Hilb^n(\C^2,0)$ is preserved by this action, so one indeed has an action of $SL(2)$ in the reduced homology $\overline{\HHH}(\beta)$ if $\beta$ closes up to a knot.

\subsection{Approaches to the proof}

In this subsection we outline some of the approaches to the proof of Conjecture \ref{conj: GNR}. The first, and the only complete at the moment is realized by Oblomkov and Rozansky \cite{OR1,OR2,OR3,OR4,OR5,OR6,OR7,OR8}. In short, they define a new link homology theory using matrix factorizations on a complicated algebraic variety related to the flag Hilbert scheme $\FHilb^n(\C^2)$ (or rather a certain smooth ambient space where  $\FHilb^n(\C^2)$ is cut out by certain equations). In \cite{OR1} Oblomkov and Rozansky prove that their construction satisfies braid relations and is invariant under Markov moves, so it does indeed define a link invariant. In further papers, they formalize the relation between their invariant, $\FHilb^n(\C^2)$ and $\Hilb^n(\C^2)$, and associate a sheaf on  $\Hilb^n(\C^2)$ to any braid. Finally, in \cite{OR8} they construct a functor from their category of matrix factorizations to the category of Soergel bimodules, and prove that their invariant agrees with Khovanov-Rozansky homology defined using the latter. 
We refer the reader to \cite{Olectures} for an introduction to Oblomkov-Rozansky theory.

Another approach, due to the first author, Hogancamp and Wedrich relates the derived category of the Hilbert scheme to the ``annular'' category of links in the solid torus. The proper definition of the latter \cite{GHW,GW} uses the machinery of derived categorical traces and is out of scope of these notes. Still, we would like to note that by \cite{BKR,Haiman,Haiman2} the derived category of $\Hilb^n(\C^2)$ is generated by the Procesi bundle $\CP$ and its direct summands, so it is essentially determined by the endomorphism algebra 
$$
\End(\CP)=\C[x_1,\ldots,x_n,y_1,\ldots,y_n]\rtimes S_n.
$$
On the other hand, the annular category is generated by the object $\Tr(1)$ which is the closure of the identity braid. It is proved in \cite{GHW} that
$$
\End(\Tr(1))=\C[x_1,\ldots,x_n,\theta_1,\ldots,\theta_n]\rtimes S_n,
$$
where $\theta_i$ are odd variables while $x_i$ are, as above even.
Experts in geometric representation theory might recognise the connection with character sheaves following \cite{Rider,RR}. The similarity of these results suggest a direct relation between the two categories, but the detailed proofs are yet to be completed.

Finally, let us sketch the third approach, as outlined in \cite{GNR}.
Consider the graded algebra
$$
\CCG:=\bigoplus_{k=0}^{\infty}\Hom(R,\FT^k),
$$
where $\FT=(\sigma_1\cdots \sigma_{n-1})^n$ is the full twist as above. The multiplication is given by the natural product (coming from the invertibility of $\FT$)
$$
\Hom(R,\FT^k)\otimes \Hom(R,\FT^{\ell})\to \Hom(R,\FT^{k+\ell})
$$
Recall that by Theorem \ref{th: full twist} we have 
$$
\Hom(R,\FT^k)=J^k/(y)J^k,\ J=\cap_{i\neq j}(x_i-x_j,y_i-y_j).
$$
so in fact by Theorem \ref{th: isospectral}(b) the graded algebra $\CCG$ coincides with the graded coordinate algebra of $X_n(\C^2,\C)$. Given an arbitrary braid $\beta$, we can consider a graded module 
$$
\bigoplus_{k=0}^{\infty}\Hom(R,\beta\cdot \FT^k),
$$
over the graded algebra $\CCG$, and hence a sheaf on $\Proj\  \CCG=X_n(\C^2,\C)$. By pushing forward along $q$, we obtain a sheaf on $\Hilb^n(\C^2)$. In \cite{GNR}, this construction is expected to lift to a dg functor 
$$
\CK^b(\SBim_n)\to D^b(\Hilb^n)
$$
satisfying Conjecture \ref{conj: GNR}. As with the previous approach, there are significant homological algebra subtleties that have to be overcome to complete the proof.

\begin{example}
\label{ex: graded algebra 2 strands}
Let us compute the graded algebra $\CCG$ for $n=2$. Recall that $T(2,2k)$ correspond to the following complexes of Soergel bimodules:
$$
\begin{array}{l}
T(2,0)=R\\
T(2,2)=[B\to B\to R]\\
T(2,4)=[B\to B\to B\to B\to R]\\
\ldots\\
\end{array}
$$
and
$$
\begin{array}{l}
\Hom(R,T(2,0))=R\\
\Hom(R,T(2,2))=[R\xrightarrow{0} R\xrightarrow{x_1-x_2} R]\\
\Hom(R,T(2,4))=[R\xrightarrow{0} R\xrightarrow{x_1-x_2} R\xrightarrow{0} R\xrightarrow{x_1-x_2}R]\\
\ldots\\
\end{array}
$$
In particular, $\Hom(R,T(2,2))$ has two generators $z$ and $w$ and one relation $w(x_1-x_2)=0$. One can check that $\Hom(R,T(2,4))$ has generators $z^2,zw,w^2$ and no more new relations (that is, $w(x_1-x_2)=0$ implies $zw(x_1-x_2)=w^2(x_1-x_2)=0$). Similarly,
$\Hom(R,T(2,2k))$ has generators $z^k,z^{k-1}w,\ldots,w^k$ and all relations follow from $w(x_1-x_2)=0$. By taking the direct sum over $k$, we get 
$$
\mathcal{G}=\frac{R[z,w]}{w(x_1-x_2)=0}=\frac{\C[x_1,x_2,z,w]}{w(x_1-x_2)=0}.
$$
Here $x_1,x_2$ are in degree zero and $z,w$ are in degree 1.
\end{example}

\begin{example}
Let us do the same computation in the $y$-ified category. This is very similar to Example \ref{ex: graded algebra 2 strands} but now the relation reads $w(x_1-x_2)=z(y_1-y_2)$, see Example \ref{ex: yified homology two strands} for details.
We get the graded algebra
$$
\mathcal{G}_{y}=\frac{R[y][z,w]}{w(x_1-x_2)=z(y_1-y_2)}=\frac{\C[x_1,x_2,y_1,y_2,z,w]}{w(x_1-x_2)=z(y_1-y_2)}
$$
where $x_1,x_2,y_1,y_2$ are in degree zero and $z,w$ are in degree 1. Note that $\mathrm{Proj}\ \mathcal{G}_y$ is the blowup of $(\C^2)^2$ along the diagonal $\{x_1=x_2,y_1=y_2\}$ which is isomorphic to the isospectral Hilbert scheme $X_2$.
\end{example}

\end{document}